\documentclass[a4paper,11pt]{amsart}

\usepackage{amsmath}
\usepackage{amssymb}
\usepackage{amsfonts}
\usepackage{amsthm}
\usepackage{graphicx}
\usepackage[french,english]{babel}
\usepackage[all]{xy}
\usepackage{pgf,tikz}
\usepackage{mathrsfs}
\usepackage{enumitem}
\usepackage{titletoc}
\usetikzlibrary{arrows}

\usepackage[a4paper]{geometry}
\usepackage{fancyhdr}
\theoremstyle {definition} \newtheorem {defin}{Definition}  [section] 
\theoremstyle {plain}  \newtheorem {thm} [defin]{Theorem}
\theoremstyle {plain}  \newtheorem {cor} [defin]{Corollary}
\theoremstyle {plain} \newtheorem {prop} [defin]{Proposition}
\theoremstyle {plain} \newtheorem {lem}[defin] {Lemma}
\theoremstyle {plain} \newtheorem {rem}[defin]{Remark}
\theoremstyle {plain} \newtheorem {ex}[defin]{Example}
\theoremstyle {plain}  \newtheorem {notat} [defin]{Notation}
\theoremstyle {plain} \newtheorem {prop defin}[defin]{Proposition-Definition}
\theoremstyle {plain}

\newcommand{\vect}{\textnormal{Vect}}

\newcommand{\rg}{\textnormal{rk}}

\newcommand{\ho}{\textnormal{Hom}}
\newcommand{\aut}{\textnormal{Aut}}
\newcommand{\iso}{\textnormal{Iso}}

\newcommand{\obj}{\textnormal{Obj}}

\newcommand{\loga}{\textnormal{Log}}

\def\wln{{\widetilde{\log}}}
\def\wMzero{{\widetilde{M_0}}}
\def\wMi{{\widetilde{M_i}}}
\def\wNzero{{\widetilde{N_0}}}
\def\wNi{{\widetilde{N_i}}}
\def\wM_\infty{{\widetilde{M_\infty}}}
\def\wN_\infty{{\widetilde{N_\infty}}}
\def\wS_1{{\widetilde{S_1}}}
\def\wF_1{{\widetilde{F_1}}}
\def\wFi{{\widetilde{F_i}}}

\def\wW_1{{\widetilde{W_1}}}
\def\ww_1{{\widetilde{w_1}}}
\def\wG_1{{\widetilde{G_1}}}
\def\wGi{{\widetilde{G_i}}}
\def\wR{{\widetilde{R}}}
\def\wmzero{{\widetilde{m_0}}}
\def\wm_\infty{{\widetilde{m_\infty}}}
\def\wf_1{{\widetilde{f_1}}}
\def\wa_0{{\widetilde{a_0}}}
\def\wa_\infty{{\widetilde{a_\infty}}}
\def\wb{{\widetilde{b}}}
\def\wa{{\widetilde{a}}}
\def\wc{{\widetilde{c}}}
\def\wz{{\widetilde{z}}}

\newcommand{\revc}{\widetilde{\mathbb{C}^\star}}
\newcommand{\sigz}{\Sigma_0}
\newcommand{\sigi}{\Sigma_\infty}

\newcommand{\pcz}{\pi^\star\mathbb{C}(z)}
\newcommand{\cz}{\mathbb{C}(z)}

\newcommand{\puip}{\phi_p}

\newcommand{\widu}{\widetilde{1}}

\newcommand{\zer}{\textnormal{Zeros}}
\newcommand{\pol}{\textnormal{Poles}}
\newcommand{\GL}{\mathrm{GL}}

\usepackage[colorlinks=true, linkcolor=blue, bookmarks=true, bookmarksopen=true]{hyperref}

\begin{document}

\sloppy

\title{A density theorem for the difference Galois groups of regular singular Mahler equations}
\author{Marina Poulet}
\address{Univ Lyon, Universit\'e Claude Bernard Lyon 1, CNRS UMR 5208, Institut Camille Jordan, F-69622 Villeurbanne, France}
\email{poulet@math.univ-lyon1.fr}
\date{October 18, 2021}

\begin{abstract}
The difference Galois theory of Mahler equations is an active research area. The present paper aims at developing the analytic aspects of this theory. We first attach a pair of connection matrices to any regular singular Mahler equation. We then show that these connection matrices can be used to produce a Zariski-dense subgroup of the difference Galois group of any regular singular Mahler equation.
\end{abstract}

\maketitle

\vspace{-1cm}

\thispagestyle{plain}
\setcounter{tocdepth}{2}
\setcounter{secnumdepth}{4}
\titlecontents{section}
[1.5em]
{}
{\contentslabel{1.3em}}
{\hspace*{-0.3em}}
{\titlerule*[1pc]{.}\contentspage}
{\addvspace{2em}\bfseries\large}
\titlecontents{subsection}
[3.8em]
{} 
{\contentslabel{2em}}
{\hspace*{-1.2em}}
{\titlerule*[1pc]{.}\contentspage}
\renewcommand{\contentsname}{Contents}
\thispagestyle{plain}
 \pdfbookmark[0]{\contentsname}{Contents}
\tableofcontents
\newpage

\section{Introduction}\label{sec:intro}

A Mahler equation of order $n\in\mathbb{N}^\star$ is a functional equation of the form
\begin{equation}\label{eq:Mahler}
a_n(z)f\left(z^{p^n}\right)+a_{n-1}(z)f\left(z^{p^{n-1}}\right)+\cdots+a_0(z)f(z)=0
\end{equation}
with $a_0(z),\ldots,a_n(z)\in\mathbb{C}(z)$, with $a_0(z)a_n(z)\neq 0$ and with $p$ an integer greater than or equal to $2$. The solutions of these equations are called Mahler functions. The study of these functions began with the work of Mahler on the algebraic relations between their special values in \cite{Mah1,Mah2,Mah3}. This work has been a source of inspiration for many authors, for example, Adamczewski and Bell \cite{AdamBell17}, Adamczewski and Faverjon \cite{AF18}, Brent, Coons and Zudilin \cite{BCZ15}, Chyzak, Dreyfus, Dumas and Mezzarobba \cite{CDDM18}, Fernandes \cite{Fer19}, Loxton \cite{Lox84}, Loxton and van der Poorten \cite{LoxPoort77}, Nishioka \cite{Nish96}, Pellarin \cite{Pel11}.
A few years ago, Philippon in \cite{Phi15}, then Adamczewski and Faverjon in \cite{AF17}, showed that the algebraic relations between special values of Mahler functions come from algebraic relations between the functions; hence the importance of the study of algebraic relations between Mahler functions. Moreover, there is an increased interest in these functions because of their connection with automata theory (see for instance \cite{AS92, Bec, Co68, MenFr80} for more details).

The algebraic relations between Mahler functions are reflected by certain difference Galois groups. This is one of the reasons why the difference Galois theoretic aspects of the theory of Mahler equations have been the subject of many recent works. For more details on this difference Galois theory and applications see for instance \cite{ADH21, DHR18, Phi15, Roq18, VDPS97}. These articles focus on the algebraic aspects of the difference Galois theory of Mahler equations. The analytic aspects have not yet been studied so much. The aim of the present paper is to start to fill this gap: our main objective is to introduce for any Mahler equation a Zariski-dense subgroup of analytic nature of its difference Galois group, which is an algebraic group. 

Our result is inspired by a celebrated density theorem due to Schlesinger for differential equations (see \cite{Sch1895}) and by its generalizations to ($q$-)difference equations discovered by Etingof and developed further by Duval, Ramis, Sauloy, Singer, van der Put. More precisely, the density theorem of Schlesinger ensures that the monodromy of a differential equation with regular singular points is Zariski-dense in its differential Galois group. This theorem was transposed to regular $q$-difference equations by Etingof (see \cite{etingof}) in $1995$ and later to regular singular $q$-difference equations by Sauloy on the one hand (see \cite{Sau03}) and by van der Put and Singer on the other hand (see \cite{VDPS97}). The key to the extensions of Schlesinger's theorem to $q$-difference equations is a connection matrix introduced by Birkhoff playing the role of the monodromy in the differential case. Roughly speaking, this matrix connects $0$ to $\infty$: given a $q$-difference equation, we associate two bases of solutions, one at $0$, the other one at $\infty$ and the connection matrix reflects the linear relations between these two bases of solutions. Likewise, for Mahler equations, the points $0$ and $\infty$ play a particular role because they are fixed points of $z\mapsto z^p$. If a Mahler equation is regular singular at $0$ and $\infty$, we can also attach two bases of solutions, one at $0$ and another one at $\infty$. However, a problem arises: these two bases of solutions can not be connected in general. Indeed, the basis of solutions at $0$ consists of meromorphic functions on the open unit disk while the basis of solutions at $\infty$ consists of meromorphic functions on the complement of the closed unit disk and, by a theorem due to Rand\'e (see \cite{BCR13,Ran92}), the unit circle is in general a natural boundary for these functions. In order to overcome this problem, our key idea is to use the point $1$, which is also a fixed point of $z\mapsto z^p$: we consider three bases of solutions, at $0$, $1$ and $\infty$ and we connect $0$ to $1$ on the one hand and $1$ to $\infty$ on the other hand. In this way, we attach two connection matrices to any regular singular Mahler equation. This allows us to construct particular elements of the Galois groupoid $G$ of the category of regular singular Mahler equations. We will show that these elements generate a Zariski-dense subgroupoid of $G$. We hope that this density theorem will have theoretical applications. For instance, it is expected to be useful for the inverse Galois problem as was the case with the $q$-difference equations (see the resolution of this inverse problem by Etingof \cite{etingof} and by van der Put and Singer \cite{VDPS97}). We hope that it will also give us a better grasp of the structure of the Tannakian Galois group of the category of regular singular Mahler equations.
\newline

We shall now describe more precisely the content of this paper. Section \ref{part1} contains prerequisites and first results about Mahler equations at the points $0$, $1$ and $ \infty$. Locally at $1$, we can use known results about $q$-difference equations, see for instance \cite{Sau00}. Indeed, locally at $1$, a Mahler equation can be seen as a $q$-difference equation with the change of variables $z=\exp(u)$. Locally, to go back to the Mahler system, we can use the principal value $\loga$ but for global considerations, we use a logarithm which is holomorphic on the universal cover $\revc$ of $\mathbb{C}^\star$. 

In Section \ref{part2}, we introduce the category $\mathcal{E}_{sf}$ of strictly Fuchsian Mahler systems and the category $\mathcal{E}_{rs}$ of regular singular Mahler systems at $0$, $1$ and $\infty$. We prove that these categories are equivalent categories.

Section \ref{part3} and Section \ref{part4} deal with the local and global Galois theory respectively. In Section \ref{part3}, we introduce the local categories $\mathcal{E}_{rs}^{(i)}$ with $i\in\lbrace 0,1,\infty\rbrace$. These categories are the localisation of $\mathcal{E}_{rs}$ at the point $i$ in the sense that we consider morphisms defined locally at $i$ (instead of morphisms with coefficients in $\mathbb{C}(z)$ for $\mathcal{E}_{rs}$). We prove that $\mathcal{E}_{rs}^{(i)}$ is equivalent to a simpler category $\mathcal{P}^{(i)}$ which is a neutral Tannakian category over $\mathbb{C}$. The local Galois group at $i$ is the Galois group of $\mathcal{P}^{(i)}$. We describe the local Galois groups: for $i=0,\infty$, we prove that the local Galois group at $i$ is $\mathbb{Z}^{alg}$, the Galois group of the category of finite dimensional complex representations of $\mathbb{Z}$ and for $i=1$, the local Galois group at $1$ is a subgroup of $\mathbb{Z}^{alg}$, described in Corollary \ref{cor:local_Galois_1}. We also describe the local Galois groupoids at $i$, denoted by $G_i$, $i=0,1,\infty$. In particular, we know Zariski-dense subgroupoids of $G_0$, $G_\infty$ (see Sections \ref{sec:repres_of_Z}, \ref{sec:GaloisGroupoid_0}) and of $G_1$ (see Section \ref{sec:GaloisGroupoid_1}). The link between the local Galois group and the local Galois groupoid is the following: let $\mathcal{X}$ be an object of $G_i$, the group of morphisms of $G_i$ from $\mathcal{X}$ to $\mathcal{X}$ is (isomorphic to) the local Galois group at $i$. These local data will be glued together in the global theory.

Section \ref{part4} deals with the global categories. We introduce the category $\mathcal{C}_{rs}$ of connection data, constructed from pairs of connection matrices attached to regular singular Mahler equations. We prove that $\mathcal{E}_{rs}$ and $\mathcal{C}_{rs}$ are equivalent Tannakian categories. More precisely, $\mathcal{C}_{rs}$ is a neutral Tannakian category over $\mathbb{C}$ equipped with fibre functors denoted by $\omega_0$, $\omega_\infty$, $\omega_1^{(\wa)}$, $\wa\in\revc\setminus\lbrace (1,0)\rbrace$.    

In Section \ref{part5}, we construct elements of the Galois groupoid $G$ of the category of regular singular Mahler equations. Roughly speaking, $G$ contains the local data given by the local Galois groupoids $G_0$, $G_1$ and $G_\infty$ and we want to make links between them thanks to the pair of connection matrices $\left(\wMzero,\wM_\infty \right)$. Let $i\in\lbrace 0,\infty\rbrace$ and $\widetilde{E_i}$ be the set of singularities of $\widetilde{M_i}$. We prove that the evaluation of $\widetilde{M_i}$ at a point $\widetilde{z_i}\not\in \widetilde{E_i}$ is a morphism of $G$ which make the link between $G_i$ and $G_1$, we denoted it by $\Gamma_{i,\widetilde{z_i}}$. In order to have a control on the singularities, we consider the Tannakian subcategory $\mathcal{C}_{E_0,E_\infty}$ of $\mathcal{C}$ which contains pairs of connection matrices whose singularities are in the sets $\widetilde{E_0}$ and $\widetilde{E_\infty}$ respectively. Then, we prove the main result of this paper:

\vspace{0.3cm}

\textbf{Theorem} \ref{thm:density_regsing}. 
The local Galois groupoids $G_0, G_1, G_\infty$ and the elements
$\Gamma_{0,\widetilde{z_0}}$, $\Gamma_{\infty,\widetilde{z_\infty}}$ for all $\widetilde{z_0}\not\in\widetilde{E_0}$, $\widetilde{z_\infty}\not\in\widetilde{E_\infty}$ generate a Zariski-dense subgroupoid $H$ of the Galois groupoid $G$ of $\mathcal{C}_{E_0,E_\infty}$.

\addtocontents{toc}{\setcounter{tocdepth}{0}}
\section*{Notations}
In this paper, we fix an integer $p$ greater than or equal to $2$.

We denote by $\mathbb{C}[[z]]$ the ring of formal power series with complex coefficients in the indeterminate $z$ and by $\mathbb{C}\lbrace z\rbrace$ the ring of all power series in $\mathbb{C}[[z]]$ that are convergent in some neighborhood of the origin. We denote by $\mathbb{C}((z))$ the fraction field of $\mathbb{C}[[z]]$ (which is called the field of formal Laurent series) and by $\mathbb{C}(\lbrace z\rbrace)$ the fraction field of $\mathbb{C}\lbrace z\rbrace$. In a similar way, $\mathbb{C}\left[\left[\frac{1}{z}\right]\right]$ is the ring of formal power series with complex coefficients in $1/z$ and $\mathbb{C}\left(\left(\frac{1}{z}\right)\right)$ is its fraction field. We denote by $\mathbb{C}\left\lbrace\frac{1}{z} \right\rbrace$ the convergent power series in $1/z$ and $\mathbb{C}\left(\left\lbrace\frac{1}{z} \right\rbrace\right)$ its fraction field. Similarly, $\mathbb{C}\lbrace z-1\rbrace$ is the ring of all power series that are convergent in some neighborhood of $1$ and $\mathbb{C}\left(\lbrace z-1\rbrace\right)$ is its fraction field. If $E$ is an open set of a Riemann surface, we denote by $\mathcal{M}\left(E\right)$ the field of meromorphic functions on $E$. We consider the universal cover of $\mathbb{C}^\star$, denoted by $\revc$, and let $\widu$ be the point $\left(1,0\right)$ of this Riemann surface. 

If $f$ is a function then $\zer(f)$ (respectively $\pol(f)$) is the set of zeros (respectively of poles) of $f$. If $M=\left(m_{i,j}\right)$ is a matrix then $\pol(M)$ is the union of the sets $\pol\left(m_{i,j}\right)$.

We denote by $\vect^f_{\mathbb{C}}$ the finite dimensional $\mathbb{C}$-vector spaces. If $M$ is a matrix, $\rg(M)$ is its rank, that is the dimension of the vector space generated by its columns.

We denote by $D(0,1)$ the open unit disk and by $\overline{D}(0,1)$ the closed unit disk.

If $\mathcal{C}$ is a category, $\obj\left(\mathcal{C}\right)$ and $\ho\left(\mathcal{C}\right)$ denote respectively the objects and the morphisms of the category $\mathcal{C}$. If $A\in\obj\left(\mathcal{C}\right)$ and $B\in\obj\left(\mathcal{C}\right)$, we denote by $\ho_{\mathcal{C}}\left(A,B\right)$ the class of morphisms in the category $\mathcal{C}$ from $A$ to $B$. If $u\in\ho_{\mathcal{C}}\left(A,B\right)$, we write $u:A\rightarrow B$. If $G$ is a groupoid and $\alpha,\beta\in \obj\left(\mathcal{C}\right)$ then $G\left(\alpha,\beta\right)$ denotes the set of morphisms of $G$ from $\alpha$ to $\beta$.

\addtocontents{toc}{\setcounter{tocdepth}{2}}

\section{Preliminary Results}\label{part1}
\subsection{Difference equations, difference systems and difference modules}

\subsubsection{Generalities}\label{sec:generalities}\leavevmode\par
For more details, the reader is referred to \cite{VDPS97}.

Let $K$ be a field of characteristic $0$ and $\phi : K\rightarrow K$ be an automorphism of $K$. We extend $\phi$ to the vector spaces of matrices with entries in $K$ applying $\phi$ on each entry of the matrices.

\vspace{0.4cm}
\paragraph{The category of difference modules}\leavevmode\par
A \emph{difference module} over the difference field $\left(K,\phi\right)$ is a pair $\left(M,\Phi_M\right)$ consisting of a finite dimensional $K$-vector space $M$ equipped with an automorphism 
$$\Phi_M : M \rightarrow M$$
which is $\phi$-linear, that is
$$\forall\lambda\in K, \forall x,y\in M, \Phi_M (x+\lambda y)=\Phi_M (x)+\phi\left(\lambda\right)\Phi_M\left(y\right).$$

The objects of the \emph{category of difference modules} are the difference modules $\left(M,\Phi_M\right)$ and the morphisms $f$ from the difference module $\left(M,\Phi_M\right)$ to the difference module $\left(N,\Phi_N\right)$
are the $K$-linear maps $f : M\rightarrow N$ such that $$\Phi_N\circ f=f\circ\Phi_M.$$
 Let 
$$C_K=K^\phi=\lbrace x\in K\mid \phi(x)=x\rbrace$$
be the \emph{constant field}. According to \cite{VDPS97}, the category of difference modules is a $C_K$-linear rigid abelian tensor category. The functor which forgets the difference structure, that is,
$$\begin{cases}
\left(M,\Phi_M\right) \leadsto  M \\
f  \leadsto  f
\end{cases}
$$
from the category of difference modules to the category of finite dimensional $K$-vector spaces $\vect^f_K$ is an exact faithful $C_K$-linear tensor functor, that is, it is a fibre functor with values in $K$. Therefore, the category of difference modules is a Tannakian category over $C_K$.

\begin{rem}
If $C_K$ is algebraically closed, from \cite[Section 1.4]{VDPS97}, the category of difference modules is a neutral Tannakian category over $C_K$. By the theory of Tannakian categories, it ensures that this category is equivalent to the category of representations of an affine group scheme.
\end{rem}

For a better understanding of what follows, we give here more details on the tensor structure of the category of difference modules. Let $\left(M,\Phi_M\right)$ and  $\left(N,\Phi_N\right)$ be two objects of this category. Their tensor product is $\left(M\otimes_K N,\Phi_{M\otimes_K N}\right)$ where $\Phi_{M\otimes_K N}$ is the automorphism 
$$
\begin{array}{rccl}
\Phi_{M\otimes_K N}: & M\otimes_K N & \rightarrow & M\otimes_K N \\
 & m\otimes n & \mapsto & \Phi_M(m)\otimes \Phi_N(n).
\end{array}
$$
The unit object is the $1$-dimensional $K$-vector space $\mathbf{1}=Ke$ equipped with the $\phi$-linear automorphism $\Phi_{\mathbf{1}} : e\rightarrow e$. The ring of endomorphisms of $\mathbf{1}$ is $C_K$. The internal Hom is 
$$\underline{\text{Hom}}\left(\left(M,\Phi_M\right),\left(N,\Phi_N\right)\right)=\left(\text{Hom}_K\left(M,N\right),\Phi_{\text{Hom}_K\left(M,N\right)}\right) $$
where $\text{Hom}_K\left(M,N\right)$ denotes the $K$-vector space of $K$-linear morphisms from $M$ to $N$ and $\Phi_{\text{Hom}_K\left(M,N\right)}$ is the $\phi$-linear automorphism 
$$
\begin{array}{rccl}
\Phi_{\text{Hom}_K\left(M,N\right)} : & \text{Hom}_K\left(M,N\right) & \rightarrow & \text{Hom}_K\left(M,N\right) \\
 & \sigma & \mapsto & \Phi_N\circ\sigma\circ\Phi_M^{-1}.
\end{array}
$$
The dual of the object $\left(M,\Phi_M\right)$ is the object 
$\underline{\text{Hom}}\left(\left(M,\Phi_M\right),\left(\mathbf{1},\Phi_{\mathbf{1}}\right)\right)$ that is $$\left(\text{Hom}_K\left(M,K\right),\sigma\in\text{Hom}_K\left(M,K\right) \mapsto\sigma\circ\Phi_M^{-1}\in\text{Hom}_K\left(M,K\right)\right).$$

We notice that the category of difference modules is equivalent to its full subcategory whose objects are the difference modules $\left(K^n,\Phi_{K^n}\right)$, $n\in\mathbb{N}$. We denote by \textit{Diff$\left(K,\phi\right)$} this subcategory. Thus, the category \textit{Diff$\left(K,\phi\right)$} is also a Tannakian category over $C_K$ with the forgetful functor as a fibre functor with values in $K$. In what follows, we identify the morphisms of \textit{Diff$\left(K,\phi\right)$} from $\left(K^m, \Phi_{K^m}\right)$ to $\left(K^n,\Phi_{K^n}\right)$ with matrices in $\mathcal{M}_{n,m}(K)$. We choose the following order on the tensor product of two ordered bases of $K^{n_1}$ and $K^{n_2}$, $n_1,n_2\in\mathbb{N}^\star$:
$$
\begin{array}{rcl}
\lbrace 1,...,n_1\rbrace\times \lbrace 1,...,n_2\rbrace & \rightarrow & \lbrace 1,...,n_1n_2\rbrace\\
\left(i_1,i_2\right) & \mapsto & i_2+n_2(i_1-1).
\end{array}
$$
It gives isomorphisms 
$$
K^{n_1}\otimes_K K^{n_2}\rightarrow K^{n_1n_2}\quad \mbox{and}\quad \mathcal{M}_{n_1,m_1}(K)\otimes_K\mathcal{M}_{n_2,m_2}(K)\rightarrow\mathcal{M}_{n_1n_2,m_1m_2}(K).
$$ 
More precisely, if $A=\left(a_{i,j}\right)\in\mathcal{M}_{n_1,m_1}\left(K\right)$ and $B\in\mathcal{M}_{n_2,m_2}\left(K\right)$ then
$$A\otimes B=\begin{pmatrix}
a_{1,1}B & \cdots & a_{1,m_1}B\\
\vdots & & \vdots\\
a_{n_1,1}B & \cdots & a_{n_1,m_1}B
\end{pmatrix}\in\mathcal{M}_{n_1n_2,m_1m_2}\left(K\right)$$ 
with the previous identification.
This tensor product on matrices is called the Kronecker product.  

\vspace{0.4cm}
\paragraph{Difference systems and equations}\leavevmode\par
Let $$L=a_n\phi^n+a_{n-1}\phi^{n-1}+\cdots+a_0$$
with $a_0,\ldots,a_n\in K$, $a_0a_n\neq 0.$
Any \emph{difference equation of order $n$}
$$Lf=a_n\phi^n(f)+a_{n-1}\phi^{n-1}(f)+\cdots+a_0f=0$$
can be transformed into the difference system of rank $n$
$$\phi\left(Y\right)=A_L Y$$
with $$A_L=\begin{pmatrix}
0 & 1 & 0 & \cdots & 0\\
\vdots & \ddots & \ddots & \ddots & \vdots \\
\vdots & & \ddots & \ddots & 0\\
0 & \cdots & \cdots & 0 & 1\\
-\dfrac{a_0}{a_n} & \cdots & \cdots & -\dfrac{a_{n-2}}{a_n} & -\dfrac{a_{n-1}}{a_n}
\end{pmatrix}\in \GL_n\left(K\right).$$
Conversely, by the cyclic vector lemma (see \cite[Appendix B]{HS99}), any system $$\phi(Y)=BY,\quad B\in \GL_n\left(K\right)$$ is equivalent via the equivalence
$$A\sim \left(\phi\left(F\right)\right)^{-1}AF,\quad F\in \GL_n(K)$$
to a system of the form $\phi(Y)=A_LY$. Thus, any difference system can be associated with a difference equation. In this article, we will focus on difference systems called Mahler systems but the results obtained can therefore be applied to Mahler equations.

The objects of the \emph{category of difference systems} are the pairs $\left(K^n,A\right)$ where $n\in\mathbb{N}$ is called the \emph{rank of} $\left(K^n,A\right)$ and $A\in \GL_n\left(K\right)$. The morphisms from $\left(K^{n_1},A\right)$ to $\left(K^{n_2},B\right)$ are the $R\in\mathcal{M}_{n_2,n_1}(K)$ such that $\phi(R)A=BR$. To simplify notations, we will often denote by $A$ the object $\left(K^n,A\right)$. This object is identified with the system $\phi(Y)=AY$.

\vspace{0.4cm}
\paragraph{The category of difference systems and the category of difference modules are equivalent}\leavevmode\par
We can associate the difference system
$$\phi\left(Y\right)=AY,\quad A\in \GL_n\left(K\right)$$ with the difference module $\left(K^n,\Phi_{K^n}\right)\in\obj\left(\mbox{\textit{Diff}$\left(K,\phi\right)$}\right)$ 
where $$\Phi_{K^n} : Y\in K^n \mapsto A^{-1}\phi(Y)\in K^n$$ and we associate the morphism $R: A\rightarrow B$ in the category of difference systems, introduced in the previous paragraph, with the morphism 
$$f : Y\in K^{n_1}\mapsto RY\in K^{n_2}$$
in the category \textit{Diff$\left(K,\phi\right)$}.
This provides an equivalence of abelian categories between the categories of difference systems and \textit{Diff$\left(K,\phi\right)$}. 
We want to obtain an equivalence of tensor categories. Therefore, we define a tensor product on the category of difference systems as follows, using the Kronecker product: if $\left(K^{n_1},A\right)$ and $\left(K^{n_2},B\right)$ are two objects then 
$$\left(K^{n_1},A\right)\otimes\left(K^{n_2},B\right)=\left(K^{n_1n_2},A\otimes B \right)$$
and if $R_1$ et $R_2$ are two morphisms, their tensor product is $R_1\otimes R_2$.
The above construction provides an equivalence of abelian tensor categories between the category of difference systems and the category \textit{Diff$\left(K,\phi\right)$}. Therefore,  the category of difference systems is a $C_K$-linear rigid abelian tensor category. It is a Tannakian category over $C_K$ with the forgetful functor $\left(K^n,A\right)\leadsto K^n$, $R\leadsto R$. 

As previously for the category of difference modules, let us explain the tensor structure inherited by the category of difference systems. The tensor product is defined thanks to the Kronecker product, as explained above. The unit object is $(K,1)$. Let $\left(K^{n_1},A\right)$ and $\left(K^{n_2},B\right)$ be two objects of this category, the internal Hom is
$$\underline{\text{Hom}}\left(\left(K^{n_1},A\right),\left(K^{n_2},B\right)\right)=\left(K^{n_1n_2},D_{A,B}\right) $$ where $D_{A,B}$ is the inverse of the matrix representing the $C_K$-linear map $$Y\in\mathcal{M}_{n_2,n_1}\left(K\right)\mapsto B^{-1}\phi(Y)A$$ 
in the canonical basis of $\mathcal{M}_{n_2,n_1}\left(K\right)$.
The dual of the object $\left(K^{n_1},A\right)$ is 
$$\underline{\text{Hom}}\left(\left(K^{n_1},A\right),\left(K,1\right)\right)=\left(K^{n_1},\left({}^t A\right)^{-1}\right). $$

\subsubsection{The category of Mahler systems with different base fields}
\label{sec:catMahler_E,E^i}
\paragraph{The category \texorpdfstring{$\mathcal{E}\left(K_{p^\infty}\right)$}{E(Kp∞)}}\leavevmode\par

The category of Mahler systems with base field $K_{p^\infty}:=\bigcup_{n\geq 0}K_{p^n}$ with $K_{p^n}=\mathbb{C}\left(z^{1/p^n}\right)$ is the particular case of categories of difference systems, introduced in Section \ref{sec:generalities}, where $K=K_{p^\infty}$ and the field automorphism of $K$ is $\phi=\puip$ defined by 
$$
\begin{array}{rccl}
\puip: & K & \rightarrow & K \\
       & f(z) & \mapsto & f\left(z^p\right).
\end{array}
$$

We will denote it by $\mathcal{E}\left(K_{p^\infty}\right)$. In this case, the constant field is $C_K=K^{\puip}=\mathbb{C}$. From Section \ref{sec:generalities}, we know that it is a neutral Tannakian category over $\mathbb{C}$.

In Proposition \ref{prop:E_Tannakian}, we prove that if the base field is $K=\mathbb{C}(z)$, it is also a neutral Tannakian category over $\mathbb{C}$. Since, in this case, the map $\puip$ is not bijective, it is not a particular case of the results of Section \ref{sec:generalities}.

\begin{rem}
The field $\mathbb{C}$ can be replaced by any algebraically closed field of characteristic zero.
\end{rem} 

\paragraph{The category \texorpdfstring{$\mathcal{E}$}{E}}\leavevmode\par

We will consider the category of Mahler systems whose base field is $K=\cz$ equipped with the injective endomorphism $\puip$, that is, the objects are the pairs $\left(\cz^n,A\right)$ where $n\in\mathbb{N}$ and $A\in \GL_n\left(\cz\right)$ and the morphisms from $\left(\cz^{n_1},A\right)$ to $\left(\cz^{n_2},B\right)$ are the $R\in\mathcal{M}_{n_2,n_1}(\cz)$ such that $\puip(R)A=BR$. We denote by $\mathcal{E}$ this category. It is a subcategory of $\mathcal{E}\left(K_{p^\infty}\right)$.

\begin{prop}\label{prop:E_Tannakian}
The category $\mathcal{E}$ is a neutral Tannakian category over $\mathbb{C}$.
\end{prop}

\begin{proof}
The category $\mathcal{E}$ is a rigid tensor subcategory of $\mathcal{E}\left(K_{p^\infty}\right)$ because it is stable under tensor product, it contains the unit object, the internal Hom and every object has a dual (see Section \ref{sec:generalities} for a description of these elements). To show that it is a Tannakian subcategory of $\mathcal{E}\left(K_{p^\infty}\right)$, the only nontrivial point which remains to be proved is the existence of kernels and cokernels in $\mathcal{E}$. Let us prove that every morphism in $\mathcal{E}$ has a kernel in $\mathcal{E}$. The existence of cokernels will follow from the fact that duality exchanges kernels with cokernels. 
Let $A,B\in\obj\left(\mathcal{E}\right)$ be respectively of rank $n_1$ and $n_2$ and $R\in\ho_{\mathcal{E}}\left(A,B\right)$. The morphism $R$ has a kernel $\left(C,K:C\rightarrow A\right)$ in the Tannakian category $\mathcal{E}\left(K_{p^\infty}\right)$. There exists $n_0\in\mathbb{N}$ such that $\puip^{n_0}\left(C\right), \puip^{n_0}\left(K\right)$ have entries in $\mathbb{C}(z)$. Let $$F_A=\left(\prod\limits_{i=1}^{n_0}\puip^{n_0-i}\left(A\right)\right)^{-1}\in \GL_{n_1}\left(\cz\right).$$ The following diagram is commutative 
$$\xymatrix{A \ar[rr]^{R} & & B\\
&&  \\
\puip^{n_0}\left(A\right)\ar[rr]^{\puip^{n_0}\left(R\right)}\ar[uu]^{F_A} & & \puip^{n_0}\left(B\right)\ar[uu]_{F_B}\\
&& \\
\puip^{n_0}\left(C\right)\ar[uu]^{\puip^{n_0}\left(K\right)}\ar@{.>}[rruu]_{0} & & }$$
and one can check that $\left(\puip^{n_0}\left(C\right), F_A\puip^{n_0}\left(K\right): \puip^{n_0}\left(C\right)\rightarrow A\right)$ is a kernel of $R$ in the category $\mathcal{E}$. 
\end{proof}

We will also consider the category of Mahler systems whose base field is $K=\mathbb{C}\left(\lbrace z\rbrace \right)$ (respectively $K=\mathbb{C}\left(\lbrace z-1\rbrace \right)$ and $K=\mathbb{C}\left(\left\lbrace\frac{1}{z} \right\rbrace\right)$) equipped with $\puip$, we will denote it by $\mathcal{E}^{(0)}$ (respectively by $\mathcal{E}^{(1)}$, by $\mathcal{E}^{(\infty )}$). These are also neutral Tannakian categories over $\mathbb{C}$. Indeed,
for the categories $\mathcal{E}^{(0)}$ and $\mathcal{E}^{(\infty )}$ we can adapt the proof of Proposition \ref{prop:E_Tannakian} and for the category $\mathcal{E}^{(1)}$, $\puip : z\in \mathbb{C}\left(\lbrace z-1\rbrace \right)\mapsto z^p$ is an automorphism of  $\mathbb{C}\left(\lbrace z-1\rbrace \right)$ so it follows from Section \ref{sec:generalities}.

\subsection{Regular singular Mahler systems}
\begin{defin}
A matrix $A$ is \emph{regular at $x\in\mathbb{C}$} if its entries are analytic functions at $x$ and $A(x)\in \GL_n\left(\mathbb{C}\right)$. The matrix $A$ is \emph{regular at $\infty$} if $A\left(1/z\right)$ is regular at $0$.
\end{defin}

\subsubsection{Neighbourhood of the point \texorpdfstring{$0$}{0}}\leavevmode\par
\label{sec:neighbourhood_0}
We consider the Mahler system 
\begin{equation}\label{eq:Mahler_at_0}
\puip\left(Y\right)=AY
\end{equation}
with $A\in \GL_n\left(\mathbb{C}\left(\lbrace z\rbrace\right)\right)$. 

\begin{defin}
The system \eqref{eq:Mahler_at_0} is \emph{strictly Fuchsian at $0$} if $A$ is regular at $0$.
\end{defin}

\begin{defin}
We say that the system \eqref{eq:Mahler_at_0} is \emph{meromorphically equivalent at $0$} to the system 
$$\puip\left(Y\right)=BY$$
with $B\in \GL_n\left(\mathbb{C}\left(\lbrace z\rbrace\right)\right)$ if there exists $T\in \GL_n\left(\mathbb{C}\left(\lbrace z\rbrace\right)\right)$, called a \emph{gauge transformation}, such that 
$$B=\puip(T)^{-1}AT.$$
\end{defin}

The motivation behind this definition is the following: if $Z$ is a solution of the Mahler system \eqref{eq:Mahler_at_0} and $T\in \GL_n\left(\mathbb{C}\left(\lbrace z\rbrace\right)\right)$ then $T^{-1}Z$ is a solution of the system $\puip\left(Y\right)=BY$.

\begin{rem}
Two systems of the form \eqref{eq:Mahler_at_0} are meromorphically equivalent at $0$ if and only if they are isomorphic in the category $\mathcal{E}^{(0)}$.
\end{rem}

\begin{defin}
The system \eqref{eq:Mahler_at_0} is \emph{regular singular at $0$} if it is meromorphically equivalent at $0$ to a strictly Fuchsian one. 
\end{defin}

From now on, we fix a submultiplicative norm $||.||$ on $\mathcal{M}_n\left(\mathbb{C}\right)$.

\begin{thm}\label{thm:fuchs_to_cst_0}
Assume that the system \eqref{eq:Mahler_at_0} is strictly Fuchsian at $0$. There exists a unique $F\in \GL_n\left(\mathbb{C}[[z]]\right)$ such that $F(0)=I_n$ and
\begin{equation}\label{eq:syst_to_cst_0}
\puip\left(F\right)^{-1}AF=A(0).
\end{equation}
We have $F\in \GL_n\left(\mathbb{C}\lbrace z\rbrace \right)$.

Moreover, if we assume that $A\in \GL_n\left(\mathcal{M}\left(D(0,1)\right)\right)$ then $F \in \GL_n\left(\mathcal{M}\left(D(0,1)\right)\right)$.
\end{thm}

\begin{proof} We use an analogue, in the Mahler case, of the Frobenius method for ordinary differential equations. Contrary to the latter one, we do not have to treat several cases depending on the eigenvalues of $A(0)$. First, we construct a formal solution, as it was done in \cite[Proposition 34]{Roq18}. Then, we prove the convergence of the formal solution. We are looking for a matrix $F\in \GL_n\left(\mathbb{C}[[z]]\right)$ such that
\begin{equation}\label{eq:syst_cst_0}
\left\lbrace
 \begin{array}{l}
  AF=\puip\left(F\right)A(0) \\
  F(0)=I_n\, .
    \end{array}
    \right.
\end{equation} 
We write 
$$F(z)=I_n+\sum\limits_{k=1}^{+\infty}F_kz^k=\sum\limits_{k=0}^{+\infty}F_kz^k$$ 
with $F_0=I_n$ and 
$$A(z)=\sum\limits_{k=0}^{+\infty}A_kz^k.$$ 
From \eqref{eq:syst_cst_0}, we have 
$$\left\lbrace
 \begin{array}{l}
  \forall k\in\mathbb{N}, p \nmid k,\, \sum\limits_{j=0}^kA_jF_{k-j}=0 \\
  \forall k\in\mathbb{N}, p\mid k,\, \sum\limits_{j=0}^kA_jF_{k-j}=F_{k/p}A_0\, .
    \end{array}
    \right.$$
Thus, this formal solution is defined by
 $$\left\lbrace
\begin{array}{l}
 F_0=I_n \\
 \forall k\in\mathbb{N}, p \nmid k,\, F_k=-A_0^{-1}\left(\sum\limits_{j=1}^kA_jF_{k-j}\right)  \\
 \forall i\in\mathbb{N}^\star, \, F_{pi}=A_0^{-1}F_iA_0-A_0^{-1}\left(\sum\limits_{j=1}^{pi}A_jF_{pi-j}\right).
\end{array}
    \right.$$ 
    
Now, we have to prove the convergence of the formal solution. We write $f_k:=||F_k||$, $a_k:=||A_k||$ and $b:=||A_0^{-1}||$. We introduce the sequence $\left(\overline{f_n}\right)_{n\geq 0}$ defined by
 $$\left\lbrace
\begin{array}{l}
 \overline{f_0}=f_0 \\
 \forall n\in\mathbb{N}^\star,\, \overline{f_n}=b\sum\limits_{i=0}^{n-1}\left( a_{n-i}+a_0\right)\overline{f_i}\, .
\end{array}
    \right.$$ 
By induction, we can show that for all $k\in\mathbb{N}$, $0\leq f_k\leq\overline{f_k}$. We have
$$\overline{S}(z):=\sum\limits_{j=0}^{+\infty}\overline{f_j}z^j=f_0+b\sum\limits_{j=1}^{+\infty}\sum\limits_{i=0}^{j-1}\left(a_{j-i}+a_0\right)\overline{f_i}z^j=f_0+b\sum\limits_{i=0}^{+\infty}\sum\limits_{j=i+1}^{+\infty}\left(a_{j-i}+a_0\right)\overline{f_i}z^j.$$ This implies 
$$\overline{S}(z)=f_0+b\overline{S}(z)\left(\sum\limits_{j=1}^{+\infty}a_jz^j+a_0\sum\limits_{j=1}^{+\infty}z^j\right).$$
Therefore, $\overline{S}(z)=\dfrac{f_0}{1-b\left(\sum\limits_{j=1}^{+\infty}a_jz^j+a_0\sum\limits_{j=1}^{+\infty}z^j\right)}$ and $\overline{S}$ is the expression as a power series around the center $z=0$ of this holomorphic function. Thus, $\sum\limits_{k=0}^{+\infty}f_kz^k$ has a nonzero radius of convergence. 

It remains to prove that if $A\in \GL_n\left(\mathcal{M}\left(D(0,1)\right)\right)$ then $F \in \GL_n\left(\mathcal{M}\left(D(0,1)\right)\right)$. From the equation \eqref{eq:syst_to_cst_0}, for $z\in D(0,1)$, we have
$$F(z)=A^{-1}(z)\cdots A^{-1}\left( z^{p^k}\right) F\left( z^{p^{k+1}}\right)A_0^{k+1}.$$ This is a product of matrices whose entries are meromorphic functions for a large enough $k$ so the entries of $F$ are meromorphic on $D(0,1)$.
\end{proof}

\begin{cor}\label{cor:regsing_to_cst_0}
If the Mahler system \eqref{eq:Mahler_at_0} is regular singular at $0$ and $A\in \GL_n\left(\mathbb{C}(z)\right)$ then there exist $F_0\in \GL_n\left(\mathcal{M}\left(D(0,1)\right)\right)$ and $A_0\in \GL_n\left(\mathbb{C}\right)$ such that
$$\puip\left(F_0\right)^{-1}AF_0=A_0.$$
\end{cor}

We highlight that we have $F_0\in \GL_n\left(\mathcal{M}\left(D(0,1)\right)\right)$ and, generally, the unit circle is a natural boundary because, from a result of Rand\'e in \cite[Theorems 4.2, 4.3]{Ran92} (see also \cite{BCR13}), a solution of the Mahler equation \eqref{eq:Mahler} is rational or it has the unit circle as natural boundary. 

\begin{rem}
In the category $\mathcal{E}^{(0)}$, the system \eqref{eq:Mahler_at_0} is regular singular at $0$ if and only if it is isomorphic to a Mahler system with constant entries. 
\end{rem}

\begin{rem}
If the system \eqref{eq:Mahler_at_0} is regular singular at $0$, we can reduce the study of its solutions to the study of the solutions of a Mahler system with constant entries. Indeed, $Z$ is a solution of $\puip\left(Y \right)=A_0Y$ if and only if $F_0Z$ is a solution of $\puip\left(Y\right)=AY$. Let us explain how to solve a constant Mahler system 
\begin{equation}
\label{eq:Mahler_syst_cst}
\puip(Y) = CY, \quad C\in \GL_n\left(\mathbb{C}\right).
\end{equation}
For details, see \cite[Section 5.2]{Roq18}. We consider functions $\ell$, $e_c$, $c\in\mathbb{C}^\star$ satisfying 
$$
\puip\left(\ell\right)=\ell +1 \quad \mbox{and}\quad \puip\left(e_c\right)=ce_c.
$$
For example, we can take $\ell(z)=\log\log(z)/\log(p)$ and $e_c(z)=\log(z)^{\log(c)/\log(p)}$. Let $C=C_uC_s$ be the multiplicative Dunford decomposition of $C$, that is, $C_u$ (respectively $C_s$) is unipotent (respectively semisimple) and the matrices $C_u$, $C_s$ commute. Let $P\in \GL_n\left(\mathbb{C}\right)$ be such that $P^{-1}C_sP:=D=\mbox{diag}\left(d_1,\ldots,d_n\right)$ is diagonal where we denote by $d_1,\ldots,d_n$ the entries of $D$ on the main diagonal. We write
$$
e_{C_u}=C_u^{\ell}:=\exp\left(\ell\log(C_u)\right)\quad \mbox{and} \quad e_{C_s}=P\mbox{diag}\left(e_{d_1},\ldots,e_{d_n}\right)P^{-1},
$$  
they satisfy $\puip\left(e_{C_u}\right)=C_ue_{C_u}$ and $\puip\left(e_{C_s}\right)=C_se_{C_s}$. We set $e_C:=e_{C_u}e_{C_s}$. Using commutativity conditions, we have 
$$
\puip\left(e_C\right)=Ce_C.
$$
The matrix $e_C$ is a fundamental matrix of solutions of \eqref{eq:Mahler_syst_cst}.
\end{rem}

\subsubsection{Neighbourhood of the point \texorpdfstring{$\infty$}{infinity}}\leavevmode\par
We consider the Mahler system 
\begin{equation}\label{eq:Mahler_at_infty}
\puip(Y)=AY
\end{equation}
with $A\in \GL_n\left(\mathbb{C}\left(\left\lbrace\frac{1}{z} \right\rbrace\right)\right)$. 
 
\begin{defin}
The system \eqref{eq:Mahler_at_infty} is \emph{strictly Fuchsian at $\infty$} if $A$ is regular at $\infty$.
\end{defin}
\begin{defin}
The system \eqref{eq:Mahler_at_infty} is \emph{regular singular at $\infty$} if it is meromorphically equivalent at $\infty$, that is via a gauge transformation $T\in \GL_n\left(\mathbb{C}\left(\left\lbrace\frac{1}{z} \right\rbrace\right)\right)$, to a strictly Fuchsian one. 
\end{defin}

Using the change of variables $z\mapsto 1/z$, we obtain results which are similar to those of Section \ref{sec:neighbourhood_0}. In particular, Corollary \ref{cor:regsing_to_cst_0} gives the following result.

\begin{cor}\label{cor:regsing_to_cst_infty}
If the Mahler system \eqref{eq:Mahler_at_infty} is regular singular at $\infty$ and $A\in \GL_n\left(\mathbb{C}(z)\right)$ then there exist $F_\infty\in \GL_n\left(\mathcal{M}\left(\mathbb{P}^1\left(\mathbb{C}\right)\setminus \overline{D}(0,1)\right)\right)$ and $A_\infty\in \GL_n\left(\mathbb{C}\right)$ such that
$$\puip\left(F_\infty\right)^{-1}AF_\infty=A_\infty.$$ 
\end{cor}

\subsubsection{Neighbourhood of the point \texorpdfstring{$1$}{1}} \leavevmode\par

We shall now study the regular singular Mahler equations at the point $1$. We consider the Mahler system 
\begin{equation}\label{eq:Mahler_at_1}
\puip\left(Y\right)=AY
\end{equation}
with $A\in \GL_n\left(\mathbb{C}\left(\lbrace z-1\rbrace \right)\right).$
\begin{defin}
The system \eqref{eq:Mahler_at_1} is \emph{strictly Fuchsian at $1$} if $A$ is regular at $1$.
\end{defin}

\begin{defin}
The system \eqref{eq:Mahler_at_1} is \emph{regular singular at $1$} if it is meromorphically equivalent at $1$, that is via a gauge transformation $T\in \GL_n\left(\mathbb{C}\left(\lbrace z-1 \right\rbrace\right)$, to a strictly Fuchsian one. 
\end{defin}
 
\begin{rem}
Two systems of the form \eqref{eq:Mahler_at_1} are meromorphically equivalent at $1$ if and only if they are isomorphic in the category $\mathcal{E}^{(1)}$.
\end{rem}

In order to study the system \eqref{eq:Mahler_at_1}, we transform it into the $q$-difference system 
\begin{equation}\label{eq:q_diff}
X(pu)=B(u)X(u)\quad\mbox{where}\quad B(u)=A\left(\exp(u)\right)
\end{equation}
with the change of variables $z=\exp(u)$.
Locally, to transform the $q$-difference system \eqref{eq:q_diff} into the Mahler system \eqref{eq:Mahler_at_1}, we can use the principal value $\loga$ because it is biholomorphic between a neighbourhood of $1$ and a neighbourhood of $0$. However, for global considerations, to go back to our initial system, we use the following logarithm which is holomorphic on the universal cover $\revc:=\left\lbrace\left(re^{ib},b\right)\mid r>0, b\in\mathbb{R}\right\rbrace$ of $\mathbb{C}^\star$:
$$
\begin{array}{rccl}
\wln: & \revc & \rightarrow & \mathbb{C}\\
	  & \left(re^{ib},b\right) & \mapsto & \log(r)+ib.
\end{array}
$$
It is moreover a biholomorphic function.

\begin{notat}
We define $$\pi:=\exp\circ\wln$$ and if $W$ is a matrix with meromorphic entries on $V\subset\mathbb{C}$, for all $\wz$ such that $\pi\left(\wz\right)\in V$,
$$\pi^\star W\left(\wz\right):=W\left(\pi\left(\wz\right)\right).$$
With an abuse of notation, we will also denote by $\puip$ the function 
$$
\begin{array}{rccl}
\puip: & \revc & \rightarrow & \revc\\
        & \left(re^{ib},b\right) & \mapsto & \left(re^{ib},b\right)^p
\end{array}
$$
where $\left(re^{ib},b\right)^p:=\left(r^pe^{ipb},pb\right)$.
\end{notat}

We notice that $$p\wln=\wln\circ\puip.$$
From \cite[Section 1.1]{Sau00}, we have the following lemma.

\begin{lem}
Let $B\in \GL_n\left(\mathcal{M}\left(\mathbb{C}\right)\right)$ be regular at $0$.
There exist $G\in \GL_n\left(\mathcal{M}\left(\mathbb{C}\right)\right)$ and $C_0\in \GL_n\left(\mathbb{C}\right)$ such that $$G(pz)^{-1}B(z)G(z)=C_0.$$
\end{lem}

Let us illustrate this result with an example.
\begin{ex}
\label{ex:qdiff,Mahler}
We consider the $q$-difference system 
\begin{equation}
\label{eq:qdiffceSystem}
X(pz) = B(z)X(z)\quad \mbox{with}\quad B(z)= \begin{pmatrix}
p & \exp\left(2z\right) \\
0 & \exp(z)
\end{pmatrix}.
\end{equation}
The matrix $B$ is regular at $0$. From \cite[Section 1.1]{Sau00}, to construct $G\in \GL_n\left(\mathcal{M}\left(\mathbb{C}\right)\right)$ and $C_0\in \GL_n\left(\mathbb{C}\right)$ such that $G(pz)^{-1}B(z)G(z)=C_0$, there are two steps which are similar to the Frobenius method for ordinary differential equations:
\begin{itemize}
\item We transform the system \eqref{eq:qdiffceSystem} into a non-resonant system at the origin, denoted by $Y(pz) = C(z)Y(z)$, that is a system such that $C$ is regular at $0$ and two distinct eigenvalues of $C(0)$ are not congruent modulo $p^\mathbb{Z}$. It is done using invertible matrices with constant entries $Q_j$ and `` shearing '' matrices $S_k$, that is matrices of the form
$$ S_k=\begin{pmatrix}
zI_l & 0\\ 0 & I_m
\end{pmatrix}.$$
More precisely, $X$ is written as $X=\underbrace{Q_1S_1...Q_rS_r}_{:=M}Y$ and $C(z)=M(pz)^{-1}B(z)M(z)$. Moreover, by construction, the matrix $C(0)$ has the same eigenvalues of $B(0)$ modulo $p^\mathbb{Z}$. 
In our example, since the eigenvalues of $B(0)$ are $1$ and $p$, the system \eqref{eq:qdiffceSystem} is resonant at the origin. We can take 
$$Q=\begin{pmatrix}
1 & 1/(1-p)\\
0 & 1
\end{pmatrix} \quad \mbox{and}\quad S=\begin{pmatrix}
z & 0\\ 0 & 1
\end{pmatrix}.$$ 
They satisfy $Q^{-1}B(0)Q = \mbox{diag}\left(p, 1\right)$ and if $M:= QS$, 
$$
C(z):= \left(M(pz)\right)^{-1}B(z)M(z) = \begin{pmatrix}
1 & g(z) \\ 0 & \exp(z)
\end{pmatrix} 
$$
where the function $g(z)=\frac{\exp(z)-p+(p-1)\exp(2z)}{p(p-1)z}$ is holomorphic on $\mathbb{C}$.
The system $Y(pz) = C(z)Y(z)$ is non-resonant at the origin. 
\item  There exists a unique matrix $H\in \GL_n\left(\mathcal{M}\left(\mathbb{C}\right)\right)$ such that $H(0)=I_n$ and $H(pz)^{-1}C(z)H(z)=C(0)$. The coefficients of the power series expansion of $H$ are determined thanks to a recurrence relation.
\end{itemize} 
The matrices $G:= MH$ and $C_0:= C(0)$ satisfy the required conditions.  
\end{ex}

\begin{thm}\label{thm:fuchs_to_cst_1}
Assume that the system \eqref{eq:Mahler_at_1} is strictly Fuchsian at $1$ and $A\in \GL_n\left(\mathbb{C}(z)\right)$. There exist 
$\wF_1\in \GL_n\left(\mathcal{M}\left(\revc\right)\right)$ and $C_0\in \GL_n\left(\mathbb{C}\right)$ such that
\begin{equation}
 \puip\left(\wF_1\right)^{-1}\left(\pi^\star A\right)\wF_1=C_0. 
\end{equation} 
\end{thm}

\begin{proof}
We transform the system \eqref{eq:Mahler_at_1} into the $q$-difference system \eqref{eq:q_diff}. From the previous lemma, we can take $\wF_1:=G\circ\wln$.
\end{proof}

\begin{ex}
We consider the Mahler system 
\begin{equation}
\puip(Y) = AY\quad \mbox{with}\quad A(z)= \begin{pmatrix}
p & z^2 \\
0 & z
\end{pmatrix},
\end{equation}
it is strictly Fuchsian at $1$. Using the notations of Example \ref{ex:qdiff,Mahler}, since $\pi^\star A = B\circ\wln$, the matrix $\wF_1:=G\circ\wln$ satisfies 
$\puip\left(\wF_1\right)^{-1}\left(\pi^\star A\right)\wF_1=C_0$.
\end{ex}

\begin{cor}\label{cor:regsing_to_cst_1}
If the Mahler system \eqref{eq:Mahler_at_1} is regular singular at $1$ and $A\in \GL_n\left(\mathbb{C}(z)\right)$ then there exist $\wF_1\in \GL_n\left(\mathcal{M}\left(\revc\right)\right)$ and $A_1\in \GL_n\left(\mathbb{C}\right)$ such that
$$\puip\left(\wF_1\right)^{-1}\left(\pi^\star A\right)\wF_1=A_1.$$ 
\end{cor}

\begin{rem}
In the category $\mathcal{E}^{(1)}$, the system \eqref{eq:Mahler_at_1} is regular singular at $1$ if and only if it is isomorphic to a Mahler system with constant entries. 
\end{rem}

\subsection{Meromorphic functions on (a part of) the universal cover of \texorpdfstring{$\mathbb{C}^\star$}{C*}}\leavevmode\par
One can identify meromorphic functions on $\mathbb{C}^\star$ with meromorphic functions on the universal cover of $\mathbb{C}^\star$, that is $\revc:=\left\lbrace \left(re^{ib},b\right)\mid r\in\mathbb{R}_+^{\star}, b\in\mathbb{R}\right\rbrace$, via
$$
\begin{array}{rccl}
i : & \mathcal{M}\left(\mathbb{C}^\star\right) & \hookrightarrow & \mathcal{M}\left(\revc\right)\\
      & f & \mapsto & f\circ\pi.
\end{array}
$$
We write
$$\pi^\star\mathcal{M}\left(\mathbb{C}^\star\right):=\left\lbrace f\circ\pi \mid f\in\mathcal{M}\left(\mathbb{C}^\star\right)\right\rbrace$$
the image of $i$, it consists of the $g\in\mathcal{M}\left(\revc\right)$ which are $2\pi$-invariant in $b$. This identification between $\mathcal{M}\left(\mathbb{C}^\star\right)$ and $\pi^\star\mathcal{M}\left(\mathbb{C}^\star\right)$ is compatible with $z\mapsto z^p$, that is, there is the following commutative diagram:

\begin{equation}
\xymatrix{\mathcal{M}\left(\mathbb{C}^\star\right) \ar[rr]^{f(z)\mapsto f\left(z^p\right)} \ar@{^{(}->}[d] & & \mathcal{M}\left(\mathbb{C}^\star\right) \ar@{^{(}->}[d]\\
\mathcal{M}\left(\revc\right)\ar[rr]_{f\left(\wz\right)\mapsto f\left(\wz^p\right)} & & \mathcal{M}\left(\revc\right)}
\end{equation}
that is,
$$\puip\circ\pi=\pi\circ\puip.$$
We also write 
$$\pi^\star\mathcal{M}\left(D(0,1)\right):=\lbrace f\circ\pi \mid f\in\mathcal{M}\left(D(0,1)\right) \rbrace$$
and
$$\pi^\star\mathcal{M}\left(\mathbb{P}^1\left(\mathbb{C}\right)\setminus\overline{D}(0,1)\right):=\left\lbrace f\circ\pi \mid f\in\mathcal{M}\left(\mathbb{P}^1\left(\mathbb{C}\right)\setminus\overline{D}(0,1)\right)\right\rbrace.$$

We can check that we have:
\begin{lem}~\
\begin{itemize}
\item[(i)] If $h\in\pi^\star\mathcal{M}\left(D(0,1)\right)$ then $h\circ\puip\in\pi^\star\mathcal{M}\left(D(0,1)\right)$.
\item[(ii)] If $h\in\pi^\star\mathcal{M}\left(\mathbb{P}^1\left(\mathbb{C}\right)\setminus\overline{D}(0,1)\right)$ then $h\circ\puip\in\pi^\star\mathcal{M}\left(\mathbb{P}^1\left(\mathbb{C}\right)\setminus\overline{D}(0,1)\right)$.
\end{itemize}
\end{lem}

\begin{lem}\label{lem:merom_funct_inter}
We have
\begin{equation}\label{eq:merom_inter_0}
\mathcal{M}\left(\revc\right)\cap\pi^\star\mathcal{M}\left(D(0,1)\right)=\pi^\star\mathcal{M}\left(\mathbb{C}\right),
\end{equation}

\begin{equation}\label{eq:merom_inter_infty}
\mathcal{M}\left(\revc\right)\cap \pi^\star\mathcal{M}\left(\mathbb{P}^1\left(\mathbb{C}\right)\setminus\overline{D}(0,1)\right)=\pi^\star\mathcal{M}\left(\mathbb{P}^1\left(\mathbb{C}\right)\setminus\lbrace 0\rbrace\right)
\end{equation}
and
\begin{equation}\label{eq:merom_inter_0andinfty}
\mathcal{M}\left(\revc\right)\cap\pi^\star\mathcal{M}\left(D(0,1)\right)\cap \pi^\star\mathcal{M}\left(\mathbb{P}^1\left(\mathbb{C}\right)\setminus\overline{D}(0,1)\right)=\pi^\star\mathcal{M}\left(\mathbb{P}^1\left(\mathbb{C}\right)\right)=\pcz.
\end{equation}
\end{lem}

\begin{proof}
Let us prove \eqref{eq:merom_inter_0}. It is clear that 
$$\pi^\star\mathcal{M}\left(\mathbb{C}\right)\subset\mathcal{M}\left(\revc\right)\cap\pi^\star\mathcal{M}\left(D(0,1)\right).$$
It remains to  show the reciprocal inclusion. Let $f\in\mathcal{M}\left(\revc\right)\cap\pi^\star\mathcal{M}\left(D(0,1)\right)$. There exists $g\in\mathcal{M}\left(D(0,1)\right)$ such that $f=\pi^\star g=g\circ\pi$. Since $f=g\circ\pi=g\circ\exp\circ\wln\in\mathcal{M}\left(\revc\right)$ and $\wln: \revc\rightarrow\mathbb{C}$ is a biholomorphic function then $g\circ\exp\in\mathcal{M}\left(\mathbb{C}\right)$. Using that $\exp: \mathbb{C}\rightarrow\mathbb{C}^\star$ is a surjective holomorphic function, we deduce that $g\in\mathcal{M}\left(\mathbb{C}^\star\right)$. Therefore, $f=g\circ\pi$ with 
$$g\in\mathcal{M}\left(D(0,1)\right)\cap\mathcal{M}\left(\mathbb{C}^\star\right)=\mathcal{M}\left(\mathbb{C}\right),$$ this proves the equality \eqref{eq:merom_inter_0}.

Similarly, we can prove the equality \eqref{eq:merom_inter_infty}. The equality \eqref{eq:merom_inter_0andinfty} is a consequence of \eqref{eq:merom_inter_0}, \eqref{eq:merom_inter_infty} and the fact that
$$\mathcal{M}\left(\mathbb{P}^1\left(\mathbb{C}\right)\right)=\mathbb{C}(z).$$
\end{proof}

\section{The Categories of Strictly Fuchsian and Regular Singular Systems at \texorpdfstring{$0,1$}{0,1} and \texorpdfstring{$\infty$}{∞}}\label{part2}

\subsection{Presentation of the categories}\label{sec:catMahler_E_rs}\leavevmode\par
We denote by $\mathcal{E}_{rs}$ the \emph{category of regular singular Mahler systems at $0,1$ and $\infty$}. It is the full subcategory of $\mathcal{E}$ (introduced in Section \ref{sec:catMahler_E,E^i}) whose objects are the  
$$\left(\mathbb{C}(z)^n,A\right)$$
where $A\in \GL_n(\mathbb{C}(z))$ is such that the system $\puip\left(Y\right)=AY$ is regular singular at $0$, $1$ and $\infty$. 

In what follows, to simplify the notations, we will omit $\mathbb{C}(z)^n$ in the definition of the objects of $\mathcal{E}_{rs}$.

We denote by $\mathcal{E}_{sf}$ the \emph{category of strictly Fuchsian Mahler systems at $0,1$ and $\infty$.}  
It is the full subcategory of $\mathcal{E}_{rs}$ whose objects are the matrices $A\in \GL_n(\mathbb{C}(z))$ which are regular at $0,1$ and $\infty$.

We will prove that the categories $\mathcal{E}_{rs}$ and $\mathcal{E}_{sf}$ are equivalent. This means that the inclusion $\mathcal{E}_{sf}\rightarrow\mathcal{E}_{rs}$ is essentially surjective.

\subsection{The categories \texorpdfstring{$\mathcal{E}_{rs}$}{Ers} and \texorpdfstring{$\mathcal{E}_{sf}$}{Esf} are equivalent}\label{sec:equivE_rs/E_f}

\begin{notat}\label{notat:mat_D,T}
We denote by $D_{i,u}$ the identity matrix whose entry in column $i$ and row $i$ is replaced by $u$. We denote by $T_{i,\underline{l}}$ where $\underline{l}=(l_1,\ldots,l_n)$ the identity matrix whose row $i$ is replaced by $\underline{l}$:
\begin{center}
$D_{i,u}=\begin{pmatrix}
1 & & & & & & \\
  &\ddots & & & & & \\
  &       & 1 & & & & \\
  &       &   & u & & & \\
  &       &   &   & 1 & & \\
  &       &   &   &   & \ddots & \\
  &       &   &   &   &        & 1  
\end{pmatrix}$ et 
$T_{i,\underline{l}}=\begin{pmatrix}
1 & & & & & & \\
  &\ddots & & & & & \\
  &       & 1 & & & & \\
 l_1 &  l_2     &  \cdots & l_i & \cdots & l_{n-1} &l_n \\
  &       &   &   & 1 & & \\
  &       &   &   &   & \ddots & \\
  &       &   &   &   &        & 1  
\end{pmatrix}$.
\end{center}
We notice that the inverse matrices of these matrices have the same form.
\end{notat}

From \cite[Section 1.3.1, Corollaire 1]{Sau00}, we have the following result.

\begin{lem}\label{lem:decomp_C_0*R_0}
If $M_0\in \GL_n\left(\mathbb{C}\left(\lbrace z\rbrace \right)\right)$ then $M_0$ can be written as
$$M_0=C^{(0)}R^{(0)}$$ with
\begin{itemize}
\item $C^{(0)}\in \GL_n\left(\mathbb{C}(z)\right)$ which can be written as
$$C^{(0)}=u^{-k}T_{i_1,\underline{l_1}}D_{i_1,v_1}\cdots T_{i_r,\underline{l_r}}D_{i_r,v_r} $$ where $u,v_1,\ldots,v_r\in\mathbb{C}\left(z\right)$ have a valuation at $0$ which is equal to $1$, $k\in\mathbb{N}$ and $\underline{l_1},\ldots,\underline{l_r}\in\mathbb{C}^n\setminus\lbrace 0\rbrace\, ;$
\item $R^{(0)}\in \GL_n\left(\mathbb{C}\left(\lbrace z\rbrace \right)\right)$ regular at $0$.
\end{itemize}
\end{lem}

\begin{lem}\label{lem:decomp_C_1*R_1}
If $M_1\in \GL_n\left(\mathbb{C}\left(\lbrace z-1\rbrace \right)\right)$ then $M_1$ can be written as $$M_1=C^{(1)}R^{(1)}$$ with
\begin{itemize}
\item $C^{(1)}\in \GL_n\left(\mathbb{C}(z)\right)$ regular at $0$ and $\infty$ which can be written as
$$C^{(1)}=u^{-k}T_{i_1,\underline{l_1}}D_{i_1,u}\cdots T_{i_r,\underline{l_r}}D_{i_r,u} $$ where $k\in\mathbb{N}$, $\underline{l_1},\ldots,\underline{l_r}\in\mathbb{C}^n\setminus\lbrace 0\rbrace$ and $u=\frac{z-1}{z+1}\, ;$
\item $R^{(1)}\in \GL_n\left(\mathbb{C}\left(\lbrace z-1\rbrace \right)\right)$ regular at $1$.
\end{itemize}
\end{lem}

\begin{proof}
To prove this, we adapt the proof of \cite[Section 1.3.1, Corollaire $1$]{Sau00}. Let $k\in\mathbb{N}$ be such that the entries of $u^kM_1:=M'_1$ are analytic at $1$. Let $v_1\left(\det\left(M'_1\right)\right)$ be the valuation at $1$ of $\det\left(M'_1\right)$. If $v_1\left(\det\left(M'_1\right)\right)=0$ then $M'_1$ is regular at $1$ so we can take $R^{(1)}:=M'_1$ and $C^{(1)}:=u^{-k}I_n$. 
If $v_1\left(\det\left(M'_1\right)\right)>0$, we do an induction on $v_1\left(\det\left(M'_1\right)\right)$: we show that there exists a matrix $N'_1$ whose entries are analytic at $1$ and such that $M'_1=T_{i_1,\underline{l_1}}D_{i_1,u}N'_1$ with $$v_1\left(\det\left(N'_1\right)\right)<v_1\left(\det\left(M'_1\right)\right).$$
If $v_1\left(\det\left(M'_1\right)\right)>0$ then $M'_1(1)$ is not invertible so there exists $\underline{l}=(l_1,\ldots,l_n)\in\mathbb{C}^n\setminus\lbrace 0\rbrace$ such that 
\begin{equation}
\label{eq:lm'(1)=0}
\underline{l}M'_1(1)=0\in\mathbb{C}^n.
\end{equation} 
Let $i\in [|1,n|]$ be such that $l_i\neq 0$. The entries of $T_{i,\underline{l}}M'_1$ are analytic at $1$ and, by the equality \eqref{eq:lm'(1)=0}, its entries of the row $i$ have a valuation at $1$ greater than or equal to $1$. Therefore, the entries of 
$$N'_1:=D_{i,u}^{-1}T_{i,\underline{l}}M'_1=D_{i,u^{-1}}T_{i,\underline{l}}M'_1$$
are analytic at $1$ and $v_1\left(\det\left(N'_1\right)\right)=v_1\left(\det\left(M'_1\right)\right)-1$.
% because $\det(N'_1)=\dfrac{l_{i}}{u}\det(M'_1)$.
\end{proof}

From \cite[Section 1.3.1, Corollaire 2]{Sau00}, we have the following result.

\begin{lem}\label{lem:reg_at_0/infty}
Let $C^{(0)}, C^{(\infty)}\in \GL_n\left(\mathbb{C}(z)\right)$. There exists $U\in \GL_n\left(\mathbb{C}(z)\right)$ such that $UC^{(0)}$ is regular at $0$ and $UC^{(\infty)}$ is regular at $\infty$.
\end{lem}

\begin{prop}\label{prop:ration_equiv}
If the Mahler system $\puip(Y)=AY$ is regular singular at $0,1$ and $\infty$, then it is rationally equivalent to a strictly Fuchsian system at $0,1$ and $\infty$, that is, there exists $R\in \GL_n\left(\mathbb{C}(z)\right)$ such that the system
$$\puip(Y)=BY\quad \mbox{with}\quad B:=\puip(R)^{-1}AR$$
is strictly Fuchsian at $0,1$ and $\infty$.
\end{prop}

\begin{proof}
From Corollaries  \ref{cor:regsing_to_cst_0}, \ref{cor:regsing_to_cst_infty} and Corollary \ref{cor:regsing_to_cst_1} (using here the principal value $\loga$ instead of $\wln$), there exist matrices $F_0\in \GL_n\left(\mathbb{C}\left(\lbrace z\rbrace\right)\right)$ (respectively $F_1\in \GL_n\left(\mathbb{C}\left(\lbrace z-1\rbrace\right)\right)$, $F_\infty\in \GL_n\left(\mathbb{C}\left(\left\lbrace \frac{1}{z}\right\rbrace\right)\right)$) and $A_0\in \GL_n\left(\mathbb{C}\right)$ (respectively $A_1,A_\infty\in \GL_n\left(\mathbb{C}\right)$) such that for $i=0,1,\infty$,
$$\puip\left(F_i\right)^{-1}AF_i=A_i.$$
If we show that there exists $R\in \GL_n\left(\mathbb{C}(z)\right)$ such that for $i=0,1,\infty$, $K_i:=R^{-1}F_i$ is regular at $i$ then
$$B=\puip(R)^{-1}AR=\puip\left(K_i\right)A_iK_i^{-1}$$
will be regular at $0,1$ and $\infty$. From Lemma \ref{lem:decomp_C_0*R_0}, $F_0$ can be written as  $C^{(0)}R^{(0)}$ where $C^{(0)}\in \GL_n\left(\mathbb{C}(z)\right)$ and $R^{(0)}\in \GL_n\left(\mathbb{C}\left(\lbrace z\rbrace\right)\right)$ is regular at $0$. Similarly, with the change of variables $z\mapsto 1/z$, $F_\infty$ can be written as $C^{(\infty)}R^{(\infty)}$ where $C^{(\infty)}\in \GL_n\left(\mathbb{C}(z)\right)$ and $R^{(\infty)}\in \GL_n\left(\mathbb{C}\left(\left\lbrace \frac{1}{z}\right\rbrace\right)\right)$ is regular at $\infty$. From Lemma \ref{lem:reg_at_0/infty}, there exists $U\in \GL_n\left(\mathbb{C}(z)\right)$ such that $UC^{(0)}$ and $UC^{(\infty)}$ are respectively regular at $0$ and $\infty$. The entries of the matrix $UF_1$ are meromorphic at $1$. From Lemma \ref{lem:decomp_C_1*R_1}, $UF_1$ can be written as $C^{(1)}R^{(1)}$ with $C^{(1)}\in \GL_n\left(\mathbb{C}(z)\right)$ regular at $0,\infty$ and $R^{(1)}$ regular at $1$. The matrix $R:=U^{-1}C^{(1)}$ satisfies the required properties. Indeed, $K_1=R^{-1}F_1=R^{(1)}$ is regular at $1$ and, for $i=0,\infty$, $K_i={C^{(1)}}^{-1}UC^{(i)}R^{(i)}$ is regular at $i$.
\end{proof}

\begin{prop}
\label{prop:equiv:Ers,Esf}
The categories $\mathcal{E}_{rs}$ and $\mathcal{E}_{sf}$ are equivalent.
\end{prop}
\begin{proof}
The natural embedding $\mathcal{E}_{sf}\rightarrow\mathcal{E}_{rs}$ is an equivalence of categories. Indeed, it is clearly fully faithful and, by Proposition \ref{prop:ration_equiv}, it is also essentially surjective.
\end{proof}

\section{Local Categories and Local Groupoids}\label{part3}
\subsection{Preliminary results}
\subsubsection{On morphisms between constant systems}
\begin{lem}\label{lem:morph_cst_syst_0}
Let $p\in\mathbb{N}_{\geq 2}$, $A_0\in \GL_{n_1}\left(\mathbb{C}\right)$ and $B_0\in \GL_{n_2}\left(\mathbb{C}\right)$. If $T\in\mathcal{M}_{n_2,n_1}\left(\mathbb{C}\left(\left( z\right)\right)\right)$ satisfies the equality 
$\puip\left(T\right)A_0=B_0T$
then $T\in\mathcal{M}_{n_2,n_1}\left(\mathbb{C}\right)$.
\end{lem}

\begin{proof}
We write $T(z)=\sum\limits_{k\geq N}T_kz^k$ with $T_k\in\mathcal{M}_{n_2,n_1}\left(\mathbb{C}\right)$ and $N=\min\lbrace k\in\mathbb{Z}\mid T_k\neq 0\rbrace$. Thereby, we have 
\begin{equation}\label{eq:dev_syst_cst_0}
\sum\limits_{k\geq N}T_kA_0z^{pk}=\sum\limits_{k\geq N}B_0T_kz^k.
\end{equation}
We obtain that $N=0$ because, looking at the lower degree, $pN=N$ namely $N=0$. Therefore, by the equality \eqref{eq:dev_syst_cst_0}, all the $T_k$ such that $p$ does not divide $k$ are equal to zero. Also, if $p$ divides $k$ then $B_0T_k=T_{k/p}A_0$ and we iterate it until we obtain an integer $\dfrac{k}{p^j}\in\mathbb{N}$ which is not divisible by $p$ so $T_k=B_0^{-j}T_{k/{p^j}}A_0^j=0$. Hence $T(z)=T_0\in\mathcal{M}_{n_2,n_1}\left(\mathbb{C}\right)$.  
\end{proof}

We deduce the following result.
\begin{lem}\label{lem:morph_cst_syst_infty}
Let $p\in\mathbb{N}_{\geq 2}$, $A_\infty\in \GL_{n_1}\left(\mathbb{C}\right)$ and $B_\infty\in \GL_{n_2}\left(\mathbb{C}\right)$. If $T\in\mathcal{M}_{n_2,n_1}\left(\mathbb{C}\left(\left( \frac{1}{z}\right)\right)\right)$ is such that $\puip\left(T\right)A_\infty=B_\infty T$ then $T\in\mathcal{M}_{n_2,n_1}\left(\mathbb{C}\right)$.
\end{lem}

\begin{lem}\label{lem:morph_cst_syst_1}
Let $p\in\mathbb{N}_{\geq 2}$, $A_1\in \GL_{n_1}\left(\mathbb{C}\right)$, $B_1\in \GL_{n_2}\left(\mathbb{C}\right)$ and $\wS_1$ be a matrix of size $n_2\times n_1$ with meromorphic entries at $\widu:= (1,0)\in\revc$. If $\puip\left(\wS_1\right)A_1=B_1\wS_1$ then the entries of $\wS_1$ are Laurent polynomials in $\wln$.
\end{lem}

\begin{proof}
Let $S_1:=\wS_1\circ\wln^{-1}$. Then,  
$$\forall z\in\mathbb{C},\quad S_1(pz)A_1=B_1S_1(z).$$
Since the entries of $S_1$ are meromorphic functions at $0$, by \cite[\S 2.1.3.2]{Sau03}, the entries of $S_1$ are Laurent polynomials. Therefore, the entries of $\wS_1=S_1\circ\wln$ are Laurent polynomials in $\wln$.
\end{proof}

\begin{notat}
We denote by $\mathbb{C}\left[\wln,\wln^{-1}\right]$ the ring of Laurent polynomials in $\wln$.
\end{notat}

\subsubsection{Finite dimensional complex representations of \texorpdfstring{$\mathbb{Z}$}{Z}}\leavevmode\par \label{sec:repres_of_Z}

Details and proofs can be found in \cite[Section 2.2.1]{Sau03} and in the references therein. We denote by $\mathcal{R}$ the category of finite dimensional complex representations of $\mathbb{Z}$. A representation $\rho$ of $\mathbb{Z}$ is entirely determined by $A=\rho(1)\in \GL_n\left(\mathbb{C}\right)$. Therefore, the objects of $\mathcal{R}$ can be seen as pairs $\left(\mathbb{C}^n,A\right)$ where $A\in \GL_n\left(\mathbb{C}\right)$ and the morphisms $F\in\ho_\mathcal{R}\left(\left(\mathbb{C}^n,A\right),\left(\mathbb{C}^p,B\right)\right)$ are the $F\in\mathcal{M}_{p,n}\left(\mathbb{C}\right)$ such that $FA=BF$. 
It is a neutral Tannakian category over $\mathbb{C}$ with the forgetful functor $\omega: \left(\mathbb{C}^n,A\right)\leadsto\mathbb{C}^n$, $F\leadsto F$ as fibre functor. Its Galois group is $$\mathbb{Z}^{alg}:=\aut^\otimes\left(\omega\right).$$

\begin{notat}
We consider $A\in \GL_n\left(\mathbb{C}\right)$ and we denote by $A=A_sA_u$ its multiplicative Dunford decomposition. Let $(\gamma,\lambda)\in\ho_{gr}\left(\mathbb{C}^*,\mathbb{C}^*\right)\times\mathbb{C}$. We write $$A^{(\gamma,\lambda)}:=\gamma\left(A_s\right)A_u^\lambda=A_u^\lambda\gamma\left(A_s\right)$$ where
\begin{itemize}
\item $\gamma$ acts on the eigenvalues of $A_s$: if $A_s=Q.\textnormal{diag}(c_1,\ldots,c_n).Q^{-1}$ then $$\gamma\left(A_s\right):=Q.\textnormal{diag}\left(\gamma(c_1),\ldots,\gamma(c_n)\right).Q^{-1}\, ;$$
\item $A_u^\lambda=\sum\limits_{k\geqslant 0}\binom{\lambda}{k}\left(A_u-I_n\right)^k$.
\end{itemize}
The matrix $A^{(\gamma,\lambda)}\in \GL_n\left(\mathbb{C}\right)$ defines an automorphism of $\mathbb{C}^n=\omega\left(\left(\mathbb{C}^n,A\right)\right)$. 
\end{notat}

\begin{prop}\label{prop:grp_isom_Zalg}
There is a proalgebraic group isomorphism 
$$
\begin{array}{rccl}
\ho_{gr}\left(\mathbb{C}^\star,\mathbb{C}^\star\right)\times\mathbb{C}  & \rightarrow & \mathbb{Z}^{alg}  & \\
(\gamma,\lambda) & \mapsto & \left(\mathbb{C}^n,A\right)\mapsto A^{(\gamma,\lambda)}.
\end{array}
$$
\end{prop}

\subsection{Localization at \texorpdfstring{$0$}{0} and at \texorpdfstring{$\infty$}{∞}}
\label{sec:localization0,infty}
\subsubsection{The categories \texorpdfstring{$\mathcal{E}_{rs}^{(0)}$}{Ers0} and \texorpdfstring{$\mathcal{E}_{rs}^{(\infty)}$}{Ers∞}}\leavevmode\par \label{sec:catMahler_E_rs0,P0}
The category $\mathcal{E}_{rs}^{(0)}$, respectively $\mathcal{E}_{rs}^{(\infty)}$, has the same objects as $\mathcal{E}_{rs}$ but morphisms from $A$ of rank $n_1$ to $B$ of rank $n_2$ are all $R\in \mathcal{M}_{n_2,n_1}\left(\mathbb{C}\left(\lbrace z\rbrace\right)\right)$, respectively $R\in \mathcal{M}_{n_2,n_1}\left(\mathbb{C}\left(\left\lbrace \frac{1}{z}\right\rbrace\right)\right)$, such that $\puip\left(R\right)A=BR$. From this equation satisfied by $R$, we deduce that 
$$
R\in \mathcal{M}_{n_2,n_1}\left(\mathcal{M}\left(D(0,1)\right)\right),\quad \mbox{respectively}\quad R\in \mathcal{M}_{n_2,n_1}\left(\mathcal{M}\left(\mathbb{P}^1\left(\mathbb{C}\right)\setminus\overline{D}(0,1)\right)\right).
$$

Until the end of Section \ref{sec:localization0,infty}, we consider $i\in\lbrace 0,\infty\rbrace$.

We denote by $\mathcal{P}^{(i)}$ the full subcategory of $\mathcal{E}_{rs}^{(i)}$ whose objects are the $C\in \GL_n\left(\mathbb{C}\right)$.

By Lemma \ref{lem:morph_cst_syst_0} and Lemma \ref{lem:morph_cst_syst_infty}, we deduce that: 
\begin{prop} If $R\in\ho_{\mathcal{P}^{(i)}}\left(A_i,B_i\right)$ with $A_i\in \GL_{n_1}\left(\mathbb{C}\right)$ and $B_i\in \GL_{n_2}\left(\mathbb{C}\right)$ then $R\in \mathcal{M}_{n_2,n_1}\left(\mathbb{C}\right)$. 

\end{prop}

It is clear that: 
\begin{lem}
The categories $\mathcal{R}$ and $\mathcal{P}^{(i)}$ are equivalent. 
\end{lem}

The following is an immediate consequence:
\begin{prop}
The category $\mathcal{P}^{(i)}$ is a neutral Tannakian category over $\mathbb{C}$ and the functor 
$$\begin{array}{rl}
\omega_{l,i}: & \mathcal{P}^{(i)} \rightarrow \vect^f_{\mathbb{C}}\\
 & \begin{cases}
A_i\leadsto \mathbb{C}^{n} \quad\mbox{with}\quad n \quad\mbox{the rank of}\quad A_i \\
R  \leadsto  R 
\end{cases}
\end{array}$$
is a fibre functor for $\mathcal{P}^{(i)}$.  
\end{prop}

\begin{prop}
The natural embedding $\mathcal{P}^{(i)}\rightarrow\mathcal{E}_{rs}^{(i)}$ is an equivalence of categories. 
\end{prop}

\begin{proof}
By definition, this functor is full and faithful. According to Corollaries \ref{cor:regsing_to_cst_0} and \ref{cor:regsing_to_cst_infty}, it is also essentially surjective 
\end{proof}

\begin{cor}
The category $\mathcal{E}_{rs}^{(i)}$ is a neutral Tannakian category over $\mathbb{C}$. 
\end{cor}

\subsubsection{Local Galois groupoids at \texorpdfstring{$0$}{0} and at \texorpdfstring{$\infty$}{∞}}\label{sec:GaloisGroupoid_0}

\begin{defin}
The \emph{local Galois groupoid at $i$}, denoted by $G_i$, is a groupoid with one object, the fibre functor $\omega_{l,i}: \mathcal{P}^{(i)} \rightarrow \vect^f_{\mathbb{C}}$, and such that the morphisms from $\omega_{l,i}$ to $\omega_{l,i}$ are the elements of 
$$G_i\left(\omega_{l,i},\omega_{l,i}\right):=\aut^\otimes\left(\omega_{l,i}\right).$$
The \emph{local Galois group at $i$} is the Galois group of $\mathcal{P}^{(i)}$ that is $G_i\left(\omega_{l,i},\omega_{l,i}\right)$. We also denote it by $G_i$.
 
\end{defin}
In Section \ref{sec:catMahler_E_rs0,P0}, we have seen that the categories $\mathcal{R}$, $\mathcal{P}^{(i)}$ and $\mathcal{E}_{rs}^{(i)}$ are equivalent Tannakian categories over $\mathbb{C}$. Therefore, they have the same Galois group which is $\mathbb{Z}^\text{alg}$ (see Section \ref{sec:repres_of_Z}).
As a consequence, the local Galois group at $i$ is 
$$G_i=\aut^\otimes\left(\omega_{l,i}\right)=\mathbb{Z}^\text{alg}.$$
The morphisms of the local Galois groupoid at $i$ are the elements of $\mathbb{Z}^{alg}.$

\subsection{Localization at \texorpdfstring{$1$}{1}}
\subsubsection{The category \texorpdfstring{$\mathcal{E}_{rs}^{(1)}$}{Ers1}}
\begin{notat}
We introduce 
$$
\begin{array}{rccl}
\pi_1: & \mathbb{C}\setminus\mathbb{R}^-\times\left]-\pi,\pi\right[ & \rightarrow & \mathbb{C}\setminus\mathbb{R}^-\\
 & \left( re^{it},t\right) & \mapsto & re^{it}.
\end{array}
$$
\end{notat}

The category $\mathcal{E}_{rs}^{(1)}$ has the same objects as $\mathcal{E}_{rs}$ but morphisms from $A$ of rank $n_1$ to $B$ of rank $n_2$ are all $R\in \mathcal{M}_{n_2,n_1}\left(\mathbb{C}\left(\lbrace z-1\rbrace\right)\right)$ such that 
\begin{equation}
\label{eq:satisfied_byR}
\puip\left(R\right)A=BR. 
\end{equation}

\begin{rem}
The entries of $R\circ\exp$ are meromorphic functions at $0$. From the equation \eqref{eq:satisfied_byR}, we have 
$\left(R\circ\exp\right)(pz)=B(\exp(z))\left(R\circ\exp\right)(z)A^{-1}(\exp(z))$
and the entries of the matrices $A\circ\exp$ and $B\circ\exp$ are meromorphic functions on $\mathbb{C}$. We deduce that 
$$R\circ\exp\in \mathcal{M}_{n_2,n_1}\left(\mathcal{M}\left(\mathbb{C}\right)\right).$$
\end{rem}

We denote by $\mathcal{P}^{(1)}$ the category whose objects are all $A_1\in \GL_n\left(\mathbb{C}\right)$ and morphisms from $A_1$ of rank $n_1$ to $B_1$ of rank $n_2$ are all matrices $\widetilde{R}$ of size $n_2\times n_1$ whose entries are meromorphic functions at $\widu\in\revc$ and such that 
$$\puip\left(\widetilde{R}\right)A_1=B_1\widetilde{R}.$$

By Lemma \ref{lem:morph_cst_syst_1}, we deduce that: 
\begin{prop}
If $\widetilde{R}\in\ho_{\mathcal{P}^{(1)}}\left(A_1,B_1\right)$ with $A_1$ and $B_1$ respectively of rank $n_1$ and $n_2$ then the entries of $\widetilde{R}$ are Laurent polynomials in $\wln$. 
\end{prop}

\begin{prop}
The categories $\mathcal{E}_{rs}^{(1)}$ and $\mathcal{P}^{(1)}$ are equivalent via the functor 
$$\begin{array}{rl}
\mathscr{G}: & \mathcal{P}^{(1)} \rightarrow \mathcal{E}_{rs}^{(1)} \\
 & \begin{cases}
A_1  \leadsto A_1 \\
\widetilde{R}  \leadsto  \widetilde{R}\circ\pi_1^{-1}.
\end{cases}
\end{array}$$
\end{prop} 

\begin{proof}
This functor is full and faithful. It is also essentially surjective according to Corollary \ref{cor:regsing_to_cst_1}.
\end{proof}

\begin{lem}\label{lem:isomP1/SauloyP0}
The category $\mathcal{P}^{(1)}$ and the category $\mathscr{P}^{(0)}$ introduced in \cite[Section 2.1.2]{Sau03} are isomorphic via the composition with $\wln$: 
$$\begin{array}{l}
 \mathscr{P}^{(0)} \rightarrow \mathcal{P}^{(1)} \\
 \begin{cases}
A_1  \leadsto A_1 \\
H  \leadsto  H\circ\wln.
\end{cases}
\end{array}$$
\end{lem}

According to \cite{Sau03}, the category $\mathscr{P}^{(0)}$ is a neutral Tannakian category over $\mathbb{C}$ so:

\begin{prop}
The category $\mathcal{P}^{(1)}$ is a neutral Tannakian category over $\mathbb{C}$ with the fibre functors
$$\begin{array}{rl}
\omega_{l,1}^{(\wc)}: & \mathcal{P}^{(1)} \rightarrow \vect^f_{\mathbb{C}}\\
 & \begin{cases}
A_1  \leadsto \mathbb{C}^{n} \quad\mbox{with}\quad n \quad\mbox{the rank of}\quad A_1 \\
\widetilde{R}  \leadsto  \widetilde{R}\left(\wc\right)
\end{cases}
\end{array}$$
for all $\wc\in\revc\setminus\lbrace\widu\rbrace$.
\end{prop}

\subsubsection{Local Galois groupoid at \texorpdfstring{$1$}{1}}\label{sec:GaloisGroupoid_1}

\begin{defin}
The \emph{local Galois groupoid at $1$}, denoted by $G_1$, is a groupoid such that its objects are the 
$$\omega_{l,1}^{(\wc)},\quad \wc\in\revc\setminus\lbrace\widu\rbrace$$ 
and its morphisms are the elements of $$G_1\left(\omega_{l,1}^{(\wc)},\omega_{l,1}^{(\widetilde{d})}\right)=\iso^\otimes\left(\omega_{l,1}^{(\wc)},\omega_{l,1}^{(\widetilde{d})}\right).$$ 
\end{defin}

\begin{thm}
The group $\aut^\otimes\left(\omega_{l,1}^{(\wc)}\right)$ is a subgroup (and the set $\iso^\otimes\left(\omega_{l,1}^{(\wc)},\omega_{l,1}^{(\widetilde{d})}\right)$ a subset) of $\mathbb{Z}^\text{alg}$.
Moreover, with the identification of $\mathbb{Z}^\text{alg}$ with $\ho_{gr}\left(\mathbb{C}^\star,\mathbb{C}^\star\right)\times\mathbb{C}$ given in Proposition \ref{prop:grp_isom_Zalg}, for all $\wc,\widetilde{d}\in\revc \setminus \lbrace\widu\rbrace$, we have
$$\iso^\otimes\left(\omega_{l,1}^{(\wc)},\omega_{l,1}^{(\widetilde{d})}\right)=
\left\lbrace\left(\gamma,\lambda\right)\in \mathbb{Z}^{alg}\left| \gamma(p)\wln\left(\wc\right)=\wln\left(\widetilde{d}\right)\right.\right\rbrace.$$
\end{thm}

\begin{proof}
We use the notations introduced in Lemma \ref{lem:isomP1/SauloyP0}. The category $\mathscr{P}^{(0)}$ is a neutral Tannakian category over $\mathbb{C}$ with fibre functors denoted by $\omega_{z_0}^{(0)}$, $z_0\in\mathbb{C}^\star$ (see \cite[Proposition 2.1.3.3]{Sau03}). Since, by construction, 
$\omega_{l,1}^{(\wc)}\left(H\circ\wln\right)=\omega_{\wln\left(\wc\right)}^{(0)}(H)$, the result follows from \cite[Theorem 2.2.2.1]{Sau03}.
\end{proof}

The local Galois groupoid at $1$ is a transitive groupoid because there exists a morphism between each object. We take as local Galois group at $1$, also denoted by $G_1$, a group $G_1:=G_1\left(\omega_{1,l}^{(\wc)},\omega_{1,l}^{(\wc)}\right)$ for a $\wc\in\revc\setminus\lbrace\widu\rbrace$. 

\begin{cor}
\label{cor:local_Galois_1}
With the same identification, for $\wc\in\revc\setminus\widu$, 
the local Galois group at $1$ is 
$$G_1:=G_1\left(\omega_{l,1}^{(\wc)},\omega_{l,1}^{(\wc)}\right)=\left\lbrace\left(\gamma,\lambda\right)\in \mathbb{Z}^{alg}\mid \gamma(p)=1\right\rbrace.$$ 
Therefore,
$$G_1\left(\omega_{l,1}^{(\wc)},\omega_{l,1}^{(\wc)}\right)\simeq \ho_{gr}\left(\mathbb{C}^\star/p^\mathbb{Z},\mathbb{C}^\star\right)\times\mathbb{C}:=G_{1,s}\times G_{1,u}.$$
\end{cor}

The local Galois group at $1$ obtained is the same group as the one obtained by Sauloy in \cite[\S 2.2.2.2]{Sau03}. Therefore, from \cite[Section 2.2.3]{Sau03}, we know a Zariski-dense subgroup of this group:

\begin{lem}[Elements of the unipotent component $G_{1,u}$ of $G_1$]\label{lem:unip_compo}
Let $A_1\in\obj\left(\mathcal{P}^{(1)}\right)$. Let $A_1=A_{1,u}A_{1,s}$ be its Dunford decomposition. The element $1\in\mathbb{C}$ which corresponds to $A_1\leadsto A_{1,u}\in\aut^\otimes\left(\omega_{l,1}^{(\wc)}\right)$ generates a Zariski-dense subgroup of the unipotent component $G_{1,u}$.
\end{lem}

\begin{notat}
Let $p\in\mathbb{N}_{\geq 2}$. We recall that all $z\in\mathbb{C}^\star$ can be uniquely written as $z=up^x$ where $u$ belongs to the unit circle and $x\in\mathbb{R}$. We define
$$
\begin{array}{rccl}
\gamma_1: & \mathbb{C}^\star & \rightarrow & \mathbb{C}^\star \\
            & up^x             & \mapsto  & u
\end{array}\quad\mbox{and}\quad
\begin{array}{rccl}
\gamma_2: & \mathbb{C}^\star & \rightarrow & \mathbb{C}^\star \\
            & up^x            & \mapsto     & e^{2i\pi x}.
\end{array}
$$
\end{notat}

\begin{lem}[Elements of the semi-simple component $G_{1,s}$ of $G_1$]
We denote by $$\overline{\gamma_1},\overline{\gamma_2}\in\ho_{gr}\left(\mathbb{C}^\star/p^\mathbb{Z},\mathbb{C}^\star\right)$$ the morphisms induced by $\gamma_1$ and $\gamma_2$ respectively. They generate a Zariski-dense subgroup of $G_{1,s}$.
\end{lem}

\begin{thm}
The subgroup of $\ho_{gr}\left(\mathbb{C}^\star/p^\mathbb{Z},\mathbb{C}^\star\right)\times\mathbb{C}$ whose unipotent component is $\mathbb{Z}\subset\mathbb{C}$ and whose semi-simple component is generated by $\overline{\gamma_1}$ and $\overline{\gamma_2}$ is Zariski-dense in $G_1$, the local Galois group at $1$. 
\end{thm}

\section{Global Categories}\label{part4}

\subsection{The category \texorpdfstring{$\mathcal{E}_{rs}$}{Ers} of regular singular Mahler systems at \texorpdfstring{$0,1$}{0,1} and \texorpdfstring{$\infty$}{∞}}\label{sec:E_rs_tannak}\leavevmode\par
We recall that $\mathcal{E}_{rs}$ is the full subcategory of $\mathcal{E}$ (introduced in Section \ref{sec:catMahler_E,E^i}) whose objects are the matrices $A\in \GL_n(\mathbb{C}(z))$ such that the system $\puip\left(Y\right)=AY$ is regular singular at $0$, $1$ and $\infty$ (see Section \ref{sec:catMahler_E_rs}).
We will show that $\mathcal{E}_{rs}$ is a neutral Tannakian category over $\mathbb{C}$. Since the inclusion $\mathcal{E}_{sf}\rightarrow\mathcal{E}_{rs}$ is an equivalence of categories (see Proposition \ref{prop:equiv:Ers,Esf}), it will also provide $\mathcal{E}_{sf}$ with a structure of neutral Tannakian category over $\mathbb{C}$.

\begin{prop}
The category $\mathcal{E}_{rs}$ is a neutral Tannakian category over $\mathbb{C}$.
\end{prop}

\begin{proof}
From Section \ref{sec:catMahler_E,E^i}, we know that $\mathcal{E}$ is a neutral Tannakian category over $\mathbb{C}$.
The category $\mathcal{E}_{rs}$ is a rigid tensor subcategory of $\mathcal{E}$ because it is closed under tensor operations, it contains the unit object $\mathbf{1}=1\in\mathbb{C}^\star$, the internal Hom and every object has a dual. To show that it is a Tannakian subcategory of $\mathcal{E}$, the only nontrivial point which remains to be proved is the existence of kernels and cokernels in $\mathcal{E}_{rs}$. This follows from the lemma herebelow and the fact that duality in the Tannakian category $\mathcal{E}$ exchanges kernels with cokernels.
\end{proof}

\begin{lem}\label{lem:kernels_in_E_rs}
All morphisms in the category $\mathcal{E}_{rs}$ have kernels.
\end{lem}

\begin{proof}
We recall that the categories $\mathcal{E}$, $\mathcal{E}^{(0)}$ (introduced in Section \ref{sec:catMahler_E,E^i}) and $\mathcal{E}_{rs}^{(0)}$ (introduced in Section \ref{sec:catMahler_E_rs0,P0}) are Tannakian categories. We consider the following embeddings of categories
$$\xymatrix{\mathcal{E}_{rs} \ar@{^{(}->}[r] \ar@{^{(}->}[d] & \mathcal{E} \ar@{^{(}->}[d]  \\
\mathcal{E}_{rs}^{(0)} \ar@{^{(}->}[r] & \mathcal{E}^{(0)}\, .}$$
Let $R\in\mathcal{M}_{n_2,n_1}\left(\mathbb{C}(z)\right)$ be a morphism in the category $\mathcal{E}_{rs}$ from the object $A\in \GL_{n_1}\left(\mathbb{C}(z)\right)$ to the object $B\in \GL_{n_2}\left(\mathbb{C}(z)\right)$. We know that the morphism $R$ has a kernel in the category $\mathcal{E}$, we denote it by
\begin{equation}\label{eq:kernel_E}
\left(C, K : C\rightarrow A\right).
\end{equation}
To show that this kernel is also a kernel of $R$ in the category $\mathcal{E}_{rs}$ it remains to prove that the system $\puip(Y)=CY$ is regular singular at $0$, $1$ and $\infty$. The morphism $R$ has also a kernel in the category $\mathcal{E}_{rs}^{(0)}$, we denote it by
\begin{equation}\label{eq:kernel_E_rs0}
\left(C_0, K_0 : C_0\rightarrow A\right)
\end{equation}
so the system $\puip(Y)=C_0Y$ is regular singular at $0$. The kernels \eqref{eq:kernel_E} and \eqref{eq:kernel_E_rs0} are kernels of the morphism $R$ in the category $\mathcal{E}^{(0)}$ so they are isomorphic that is there exists an invertible morphism $u_1: C_0\rightarrow C$ in the category $\mathcal{E}^{(0)}$ such that $Ku_1=K$. Since $u_1\in\ho_{\mathcal{E}^{(0)}}\left(C_0,C\right)$, the entries of $u_1$ are meromorphic functions at $0$ and $u_1$ satisfies 
$$\puip\left(u_1\right)C_0=Cu_1$$
thus the system $\puip(Y)=CY$ is also regular singular at $0$. Similarly, we can show that the system $\puip(Y)=CY$ is regular singular at $1$ and at $\infty$ so $C\in\obj\left(\mathcal{E}_{rs}\right)$ and $\left(C, K : C\rightarrow A\right)$ is also a kernel of $R$ in the category $\mathcal{E}_{rs}$. 
\end{proof}

We introduce in Section \ref{sec:cat_connections} and Section \ref{sec:catS} two categories, respectively the category $\mathcal{C}_{rs}$ of connection data and the category $\mathcal{S}_{rs}$ of solutions. 
We want to prove that $\mathcal{C}_{rs}$ and $\mathcal{E}_{rs}$ are equivalent. The category $\mathcal{S}_{rs}$ only plays an intermediate role. More precisely, we will see that $\mathcal{C}_{rs}$ and $\mathcal{E}_{rs}$ are equivalent via two functors, which are equivalence of categories, from $\mathcal{S}_{rs}$ to respectively $\mathcal{C}_{rs}$ and $\mathcal{E}_{rs}$. Thanks to $\mathcal{S}_{rs}$, we have well-defined functors which are independent of any choice (choice of a solution associated with a system, choice of a decomposition of connection matrices).

\subsection{The category \texorpdfstring{$\mathcal{C}_{rs}$}{Crs} of connection data}\label{sec:cat_connections}\leavevmode\par

We define $$\sigz:=\left\lbrace \left(re^{ib},b\right)\mid 0<r<1, b\in\mathbb{R}\right\rbrace\subset\revc$$ and 
$$\sigi:=\left\lbrace \left(re^{ib},b\right)\mid r>1, b\in\mathbb{R}\right\rbrace\subset\revc.$$ 

We introduce the category $\mathcal{C}_{rs}$ of connection data. Its objects are the
$$(A_0,A_1,A_\infty,\wMzero ,\wM_\infty)\in \GL_n(\mathbb{C})^3\times  \GL_n\left(\mathcal{M}\left(\sigz\right)\right)\times \GL_n\left(\mathcal{M}\left(\sigi\right)\right)$$ such that
$$\left\lbrace
 \begin{array}{l}
\puip\left(\wMzero\right)=A_1\wMzero A_0^{-1} \\
\puip\left(\wM_\infty\right)=A_1\wM_\infty A_\infty^{-1}
    \end{array}
    \right.$$
and such that there exist
\begin{center} 
 $\wW_1\in \GL_n\left(\mathcal{M}\left(\revc\right)\right)$, $W_0\in \GL_n\left(\mathcal{M}\left(D(0,1)\right)\right)$, $W_\infty\in \GL_n\left(\mathcal{M}\left(\mathbb{P}^1\left(\mathbb{C}\right)\setminus\overline{D}(0,1)\right)\right)$
\end{center}
such that
 $$\left\lbrace
 \begin{array}{l}
\wW_1\wMzero=\pi^\star W_0  \\
\wW_1\wM_\infty=\pi^\star W_\infty. 
    \end{array}
    \right.$$
The integer $n$ is called the \emph{rank} of the object.

Morphisms from $(A_0,A_1,A_\infty,\wMzero,\wM_\infty)\in\obj\left(\mathcal{C}_{rs}\right)$ to $(B_0,B_1,B_\infty,\wNzero,\wN_\infty)\in\obj\left(\mathcal{C}_{rs}\right)$, of rank $n_1$ and $n_2$ respectively, are the triples $$\left(S_0,\wS_1,S_\infty\right)\in\mathcal{M}_{n_2,n_1}(\mathbb{C})\times\mathcal{M}_{n_2,n_1}\left(\mathbb{C}\left[\wln,\wln^{-1}\right]\right)\times\mathcal{M}_{n_2,n_1}(\mathbb{C})$$ such that 
 $$\left\lbrace
\begin{array}{l}
S_0A_0=B_0S_0 \\
\puip\left(\wS_1\right)A_1=B_1\wS_1 \\
S_\infty A_\infty =B_\infty S_\infty \\
\wS_1\wMzero=\wNzero S_0 \\
\wS_1\wM_\infty=\wN_\infty S_\infty .
\end{array}
    \right.$$ 
    
\begin{rem}
From Lemmas \ref{lem:morph_cst_syst_0}, \ref{lem:morph_cst_syst_infty} and \ref{lem:morph_cst_syst_1}, a triple $\left(S_0,\wS_1,S_\infty\right)$ is a morphism of the category $\mathcal{C}_{rs}$ if and only if it satisfies the above system and the entries of $S_0$, $\wS_1$ and $S_\infty$ are meromorphic functions at $0$, $\widu$ and $\infty$ respectively.
\end{rem}

\subsection{The category \texorpdfstring{$\mathcal{S}_{rs}$}{Srs} of solutions}\label{sec:catS}\leavevmode\par

The objects of the category $\mathcal{S}_{rs}$ of solutions are the 
$(A_0,A_1,A_\infty,F_0,\wF_1,F_\infty)$ 
such that 
$$ A_0, A_1, A_\infty \in \GL_n(\mathbb{C}),$$
$$F_0\in \GL_n\left(\mathcal{M}\left(D(0,1)\right)\right),$$
$$\wF_1\in \GL_n\left(\mathcal{M}\left(\revc\right)\right),$$
$$F_\infty\in \GL_n\left(\mathcal{M}\left(\mathbb{P}^1\left(\mathbb{C}\right)\setminus\overline{D}(0,1)\right)\right)$$
and such that if we define 
$$
A := \left\lbrace\begin{array}{ll}
\puip(F_0)A_0F_0^{-1} & \mbox{on}\quad D(0,1) \\
\puip(F_\infty)A_\infty F_\infty^{-1} & \mbox{on}\quad \mathbb{P}^1\left(\mathbb{C}\right)\setminus\overline{D}(0,1)
\end{array}\right.
$$
and $\widetilde{A} := \puip(\wF_1)A_1\wF_1^{-1}$ on $\revc$ then $$\pi^\star A=\widetilde{A}.$$
The integer $n$ is called the \emph{rank} of the object.

Morphisms from $(A_0,A_1,A_\infty,F_0,\wF_1,F_\infty)$ to $(B_0,B_1,B_\infty,G_0,\wG_1,G_\infty)$, two objects of $\mathcal{S}_{rs}$ of rank $n_1$ and $n_2$ respectively, are the $$(R,S_0,\wS_1,S_\infty)\in\mathcal{M}_{n_2,n_1}(\mathbb{C}(z))\times\mathcal{M}_{n_2,n_1}(\mathbb{C})\times\mathcal{M}_{n_2,n_1}\left(\mathbb{C}\left[\wln,\wln^{-1}\right]\right)\times\mathcal{M}_{n_2,n_1}(\mathbb{C})$$ such that 
 $$\left\lbrace
\begin{array}{l}
S_0A_0=B_0S_0\\
\puip\left(\wS_1\right)A_1=B_1\wS_1\\
S_\infty A_\infty =B_\infty S_\infty\\
RF_i=G_iS_i\quad\textnormal{for}\quad i=0,\infty\\
\wR\wF_1=\wG_1\wS_1
\end{array}
    \right.$$
where we define $\wR:=\pi^\star R$.

\subsection{Equivalence between the categories}\label{sec:equivalence_cat}\leavevmode\par
Our goal is to obtain an equivalence between the categories $\mathcal{C}_{rs}$ and $\mathcal{E}_{rs}$.

\begin{prop}\label{prop:equivS/E_rs}
The categories $\mathcal{S}_{rs}$ and $\mathcal{E}_{rs}$ are equivalent by the functor $$\begin{array}{rl}
\mathscr{F}: & \mathcal{S}_{rs} \rightarrow \mathcal{E}_{rs} \\
 & \begin{cases}
(A_0,A_1,A_\infty,F_0,\wF_1,F_\infty)  \leadsto A \quad\mbox{defined in the category}\, \mathcal{S}_{rs}\quad(\mbox{in Section \ref{sec:catS}})\\
  \left(R,S_0,\wS_1,S_\infty\right) \leadsto R .
\end{cases}
\end{array}$$
\end{prop} 

\begin{proof}
We check that $\mathscr{F}$ is well-defined on objects. 
By the equalities $A=\puip\left(F_0\right)A_0F_0^{-1}$, $A=\puip\left(F_\infty\right)A_\infty F_\infty^{-1}$ and $\pi^\star A=\puip\left(\wF_1\right) A_1\wF_1^{-1}$, we obtain respectively that the entries of $\pi^\star A$ belong to $\pi^\star\mathcal{M}\left(D(0,1)\right)$, $\pi^\star\mathcal{M}\left(\mathbb{P}^1\left(\mathbb{C}\right)\setminus\overline{D}(0,1)\right)$ and $\mathcal{M}\left(\revc\right)$. By Lemma \ref{lem:merom_funct_inter}, $\pi^\star A\in \GL_n\left(\pcz\right)$ so $$A\in \GL_n\left(\mathbb{C}(z)\right).$$ Moreover, the system $\puip(Y)=AY$ is regular singular at $0,1,\infty$. The functor $\mathscr{F}$ is also well-defined on morphisms because $R\in\mathcal{M}_{n_2,n_1}(\mathbb{C}(z))$ and it satisfies 
$$\puip(R)A=\underbrace{\puip(R)\puip(F_0)}_{=\puip(G_0)S_0}A_0F_0^{-1}=\puip(G_0)B_0\underbrace{S_0F_0^{-1}}_{=G_0^{-1}R}=BR.$$
The functor $\mathscr{F}$ is essentially surjective because of Corollaries \ref{cor:regsing_to_cst_0}, \ref{cor:regsing_to_cst_infty}, \ref{cor:regsing_to_cst_1}. It is also fully faithful. Indeed, keeping the notations introduced in Section \ref{sec:catS}, if $R\in\ho_{\mathcal{E}_{rs}}(A,B)$ then $\left(R,G_0^{-1}RF_0,\wG_1^{-1}\wR\wF_1,G_\infty^{-1}RF_\infty\right)$ is the unique morphism of $\mathcal{S}_{rs}$ such that $$\mathscr{F}(R,S_0,\wS_1,S_\infty)=R$$ because it is uniquely determined and it satisfies 
\begin{itemize}
\item for $i=0,\infty$, $$\puip\left(S_i\right)A_i=\puip\left(G_i\right)^{-1}\puip\left(R\right)\puip\left(F_i\right)A_i=B_iG_i^{-1}\underbrace{B^{-1}\puip\left(R\right)A}_{=R}F_i=B_iS_i$$ where we used the fact that $\puip\left(G_i\right)=BG_iB_i^{-1}$ and $\puip\left(F_i\right)=AF_iA_i^{-1}$ for the second equality. By Lemmas \ref{lem:morph_cst_syst_0} and \ref{lem:morph_cst_syst_infty}, we obtain that $S_i$ has constant entries.
\item Similarly, $\puip\left(\wS_1\right)A_1=B_1\wS_1$ and by Lemma \ref{lem:morph_cst_syst_1}, the entries of $\wS_1$ are Laurent polynomials in $\wln$.
\end{itemize} 
\end{proof}

\begin{prop}
\label{prop:equiv:S/C}
The categories $\mathcal{S}_{rs}$ and $\mathcal{C}_{rs}$ are equivalent by the functor $$\begin{array}{rl}
\mathscr{G}: & \mathcal{S}_{rs} \rightarrow \mathcal{C}_{rs} \\
 & \begin{cases}
\left(A_0,A_1,A_\infty,F_0,\wF_1,F_\infty\right)  \leadsto \left(A_0,A_1,A_\infty,\wF_1^{-1}\pi^\star F_0,\wF_1^{-1}\pi^\star F_\infty\right)\\
  \left(R,S_0,\wS_1,S_\infty\right) \leadsto \left(S_0,\wS_1,S_\infty\right).
\end{cases}
\end{array}$$
\end{prop} 

\begin{proof}
We write $\wFi:=\pi^\star F_i$ for $i=0,\infty$ and $\wR:=\pi^\star R$. The functor $\mathscr{G}$ is clearly well-defined on objects. It is also well-defined on morphisms because if $$(R,S_0,\wS_1,S_\infty)\in\ho_{\mathcal{S}_{rs}}\left((A_0,A_1,A_\infty,F_0,\wF_1,F_\infty),(B_0,B_1,B_\infty,G_0,\wG_1,G_\infty) \right),$$ for $i=0,\infty$, introducing $\wMi=\wF_1^{-1}\wFi$ and $\wNi=\wG_1^{-1}\wGi$,
 $$\wS_1\wMi=\underbrace{\wS_1 \wF_1^{-1}}_{=\wG_1^{-1}\wR}\wFi=\wNi S_i.$$ 
The functor $\mathscr{G}$ is an equivalence of categories because it is:
\begin{itemize}
\item essentially surjective. Here, each object $(A_0,A_1,A_\infty,\wMzero,\wM_\infty)$ of $\mathcal{C}_{rs}$ is of the form  $G(A_0,A_1,A_\infty,F_0,\wF_1,F_\infty)$. Indeed, we set $\wF_1:=\wW_1$ and $F_0:=W_0$, $F_\infty:=W_\infty$ (notations introduced in the category of connection data, Section \ref{sec:cat_connections}).

\item fully faithful. Let $$(S_0,\wS_1,S_\infty)\in\ho_{\mathcal{C}_{rs}}\left((A_0,A_1,A_\infty,\wMzero,\wM_\infty),(B_0,B_1,B_\infty,\wNzero,\wN_\infty)\right).$$ There exists exactly one morphism $(R,S_0,\wS_1,S_\infty)\in\ho\left(\mathcal{S}_{rs}\right)$ such that $$\mathscr{G}(R,S_0,\wS_1,S_\infty)=(S_0,\wS_1,S_\infty).$$
It is clear because $R=G_iS_iF_i^{-1}$ for $i=0,\infty$ and $$\pi^\star R=\wGi S_i\wFi^{-1}=\wG_1\wS_1\wF_1^{-1}.$$ Therefore, the entries of $\pi^\star R$ belong to
$$\mathcal{M}\left(\revc\right)\cap\pi^\star\mathcal{M}\left(D(0,1)\right)\cap \pi^\star\mathcal{M}\left(\mathbb{P}^1\left(\mathbb{C}\right)\setminus\overline{D}(0,1)\right).$$ By Lemma \ref{lem:merom_funct_inter}, the entries of $\pi^\star R$ belong to $\pcz$ that is the entries of $R$ are rational functions.
\end{itemize}
\end{proof}

From the two previous propositions, we have the following result.

\begin{cor}\label{cor:equiv_global_cat}
The categories $\mathcal{C}_{rs}$ and $\mathcal{E}_{rs}$ are equivalent.
\end{cor} 

\subsection{Structure of neutral Tannakian category over \texorpdfstring{$\mathbb{C}$}{C}}\leavevmode\par

As we did for $\mathcal{E}_{rs}$ (and $\mathcal{E}_{sf}$), we equip the categories $\mathcal{C}_{rs}$ and $\mathcal{S}_{rs}$ with a tensor product on objects and morphisms. It is defined componentwise using the tensor product on matrices defined in Section \ref{sec:generalities}. The equivalences defined in Section \ref{sec:equivalence_cat} are compatible with these tensor structures. Since $\mathcal{E}_{rs}$ is a neutral Tannakian category over $\mathbb{C}$ (see Section \ref{sec:E_rs_tannak}), the categories $\mathcal{C}_{rs}$ and $\mathcal{S}_{rs}$ are also neutral Tannakian categories over $\mathbb{C}$. In Proposition \ref{prop:fibre_functors}, we explicit fibre functors for the category $\mathcal{C}_{rs}$.

\begin{prop}\label{prop:fibre_functors}
For all $\wa\in\revc\setminus\lbrace\widu\rbrace$, we consider the functors $$\omega_0,\omega_\infty, \omega_1^{\left(\wa\right)}:\mathcal{C}_{rs}\rightarrow \vect^f_{\mathbb{C}}$$ defined by:
\begin{itemize}
\item if $\mathcal{X}\in\obj\left(\mathcal{C}_{rs}\right)$ is of rank $n$ then for $i=0,\infty$, 
$$\omega_i\left(\mathcal{X}\right)=\mathbb{C}^n\quad\mbox{and}\quad\omega_1^{\left(\wa\right)}\left(\mathcal{X}\right)=\mathbb{C}^n,$$
\item if $\left(S_0,\wS_1,S_\infty\right)\in\ho\left(\mathcal{C}_{rs}\right)$ then for $i=0,\infty$, 
$$\omega_i\left(S_0,\wS_1,S_\infty\right)=S_i\quad\mbox{and}\quad \omega_1^{\left(\wa\right)}\left(S_0,\wS_1,S_\infty\right)=\wS_1\left(\wa\right).$$
\end{itemize}
These are fibre functors for the category of connection data $\mathcal{C}_{rs}$, which is a neutral Tannakian category over $\mathbb{C}$.
\end{prop}

\begin{proof}
The only non-trivial point is the faithfulness of the functors $\omega_1^{\left(\wa\right)}$, $\wa\in\revc\setminus\lbrace\widu\rbrace$. To prove it, we assume that $\wS_1\left(\wa\right)=0$ and we want to show that $\wS_1=0$. From the equality $\puip\left(\wS_1\right)=B_1\wS_1A_1^{-1}$ it follows that $\wS_1\left(\wa^{p^k}\right)=0$ for all $k\in\mathbb{Z}$. Since $\wa\in\revc\setminus\lbrace\widu\rbrace$, the entries of $\wS_1$ are Laurent polynomials in $\wln$ with an infinite number of roots, thus they are equal to $0$. 
\end{proof}

\section{Mahler Analogue of the Density Theorem of Schlesinger}\label{part5}

In order to prove (an analogue of) the density theorem of Schlesinger, the two main approaches are the use of a Picard-Vessiot theory and the use of a Tannakian duality (see for instance \cite{Sau03,VDPS97}). In \cite[Theorem 1.32]{VDPS97}, the authors proved that a Galois group in a Picard-Vessiot theory can be seen as a group of automorphisms of a suitable fibre functor $\omega$, that is the group $\aut^\otimes\left(\omega\right)$. Thus, a Galois group is a Tannakian Galois group. Moreover, two fibre functors defined on a same neutral Tannakian category are isomorphic so these two theories coincide. In what follows, we use the Tannakian approach.

\subsection{The Galois groupoid of the category \texorpdfstring{$\mathcal{C}_{rs}$}{Crs}}
\begin{defin}\label{defin:groupoid_C}
The \emph{Galois groupoid of the category of connection data $\mathcal{C}_{rs}$}, denoted by $G$, is such that:
\begin{itemize}
\item the objects are the fibre functors $\omega_0, \omega_\infty, \omega_1^{(\wa)}: \mathcal{C}_{rs}\rightarrow \vect^f_{\mathbb{C}}$ defined in Proposition \ref{prop:fibre_functors},
\item the morphisms between $\omega,\omega'\in\left\lbrace \omega_0\right\rbrace\cup\left\lbrace\omega_\infty\right\rbrace\cup\left\lbrace\omega_1^{\left(\wa\right)}\mid \wa\in\revc\setminus\lbrace\widu\rbrace\right\rbrace$ are the isomorphisms of tensor functors from $\omega$ to $\omega'$:
$$G\left(\omega,\omega'\right)=\iso^\otimes\left(\omega,\omega'\right).$$
\end{itemize}
\end{defin}

\begin{prop}
\label{prop:identify_groupoids}
The local Galois groupoids $G_0$, $G_1$ and $G_\infty$ can be identified with subgroupoids of the Galois groupoid $G$ of the category $\mathcal{C}_{rs}$.  
\end{prop}

\begin{proof}
Let us prove that $G_0$ can be identified with a subgroupoid of $G$. We can see that $\omega_0=\omega_{l,0}\mathscr{P}$ with 
$$
\begin{array}{rccl}
\mathscr{P} : & \mathcal{C}_{rs} & \rightarrow & \mathcal{P}^{(0)}\\
& \left(A_0,A_1,A_\infty,\wMzero,\wM_\infty\right)\in\obj\left(\mathcal{C}_{rs}\right) & \leadsto & A_0\\
& \left(S_0,\wS_1,S_\infty\right) & \leadsto & S_0
\end{array}
$$
the functor `` projection ''.  
The restriction of the elements of $\aut^\otimes\left(\omega_{l,0}\right)$ to $\mathcal{C}_{rs}$ gives a morphism of proalgebraic groups 
$$\mathscr{P}^\star : \aut^\otimes\left(\omega_{l,0}\right)\rightarrow\aut^\otimes\left(\omega_{0}\right).$$ 
Let us prove that it is a closed immersion of proalgebraic groups. Every object of the category of representations of $\aut^\otimes\left(\omega_{l,0}\right)$, denoted by $\mbox{Rep}\left(\aut^\otimes\left(\omega_{l,0}\right)\right)$, is isomorphic to a subquotient of an object of $\mbox{Rep}\left(\aut^\otimes\left(\omega_{l,0}\right)\right)$ of the form $\omega^{\mathscr{P}^\star}\left(\mathcal{X}\right)$ where $\mathcal{X}\in\obj\left(\mbox{Rep}\left(\aut^\otimes\left(\omega_{0}\right)\right)\right)$ and 
$$\omega^{\mathscr{P}^\star}: \mbox{Rep}\left(\aut^\otimes\left(\omega_{0}\right)\right) \rightarrow\mbox{Rep}\left(\aut^\otimes\left(\omega_{l,0}\right)\right).$$
Indeed, $\mathscr{P}$ is essentially surjective and, by Tannakian duality, the categories $\mbox{Rep}\left(\aut^\otimes\left(\omega_{0}\right)\right)$ and $\mathcal{C}_{rs}$ are equivalent and the categories $\mbox{Rep}\left(\aut^\otimes\left(\omega_{l,0}\right)\right)$ and $\mathcal{P}^{(0)}$ are equivalent. Thus, from \cite[Proposition 2.21]{DelMil}, $\mathscr{P}^\star$ is a closed immersion of proalgebraic groups from $G_0\left(\omega_{l,0},\omega_{l,0}\right)=\aut^\otimes\left(\omega_{l,0}\right)$ to $G\left(\omega_{0},\omega_{0}\right)=\aut^\otimes\left(\omega_{0}\right)$.
\end{proof}

We also denote by $G_0,G_1,G_\infty$ the subgroupoids of $G$ which are identified with $G_0,G_1,G_\infty$ in Proposition \ref{prop:identify_groupoids}. These subgroupoids give elements $\varphi_0\in G\left(\omega_{0},\omega_{0}\right)$, $\varphi_\infty\in G\left(\omega_{\infty},\omega_{\infty}\right)$ and $\varphi_1^{\left(\widetilde{c},\widetilde{d}\right)}\in G\left(\omega_1^{\left(\widetilde{c}\right)},\omega_1^{\left(\widetilde{d}\right)}\right)$ for all $\widetilde{c},\widetilde{d}\in\revc\setminus\lbrace\widu\rbrace$. In Section \ref{sec:construction_isom}, we build elements $\Gamma_{0,\wa}\in G\left(\omega_0,\omega_1^{\left(\wa\right)} \right)$ and $\Gamma_{\infty,\wb}\in G\left(\omega_{\infty},\omega_1^{\left(\wb\right)} \right)$ for all $\wa\in\sigz$, $\wb\in\sigi$, they are Galoisian isomorphisms. They can be seen as paths in the Galois groupoid $G$ connecting the point $\omega_0$ to the point $\omega_1^{\left(\wa\right)}$ and the point $\omega_{\infty}$ to the point $\omega_1^{\left(\wb\right)}$ respectively. Therefore, $G$ is a transitive groupoid. We write $\widetilde{z}_0\in\sigz$, $\widetilde{z}_{\infty}\in\sigi$ and $\widetilde{z}_1,\widetilde{z}_1'\in\revc\setminus\left(\sigz\cup\sigi\cup\lbrace\widu\rbrace\right)$. The following diagram (which is not commutative) represents the transitive groupoid $G$.

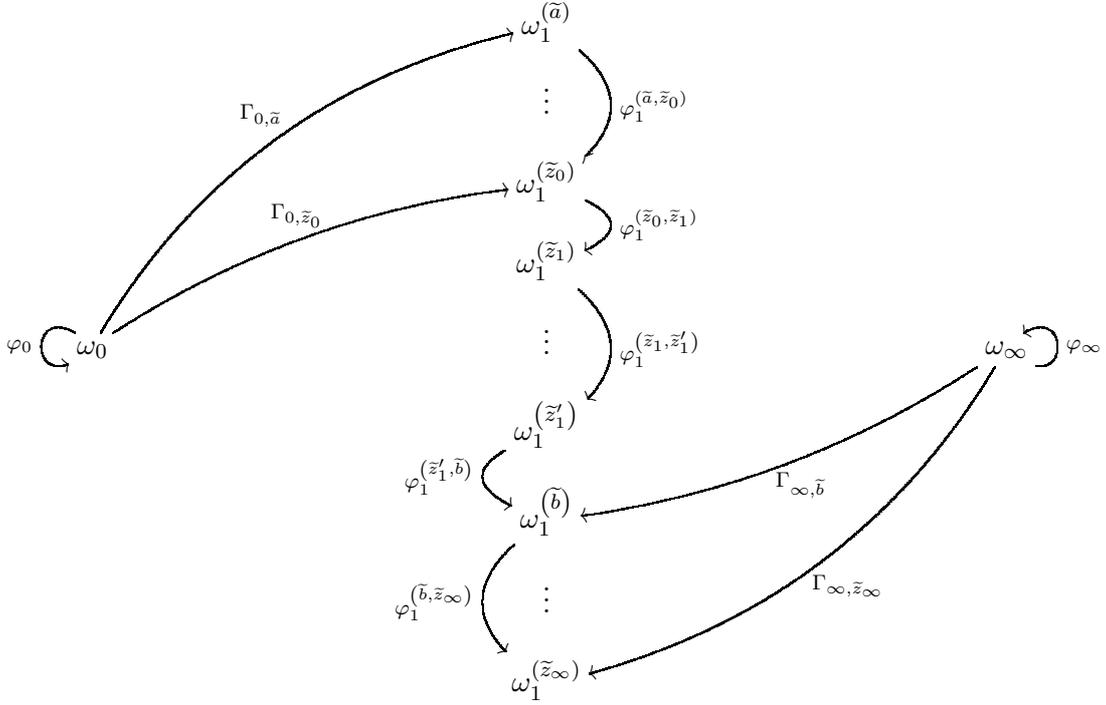
$$
\SelectTips{eu}{12}
\xymatrix @R=0.3cm @C=1.1cm{ 
   & & & & \omega_1^{\left(\wa\right)}  \ar@/^2pc/[dd]^{\varphi_1^{\left(\wa,\widetilde{z}_0\right)}} & & & & \\
   & & & & \vdots & & & & \\
   & & & & \omega_1^{\left(\widetilde{z}_0\right)} \ar@/^2pc/[d]^{\varphi_1^{\left(\widetilde{z}_0,\widetilde{z}_1\right)}} & & & & \\
   & & & & \omega_1^{\left(\widetilde{z}_1\right)}\ar@/^2pc/[dd]^{\varphi_1^{\left(\widetilde{z}_1,\widetilde{z}_1'\right)}} & & & & \\
   \omega_0\ar@(ul,dl)[]_{\varphi_0}  \ar@/^2pc/[rrrruuuu]^{\Gamma_{0,\wa}} \ar@/^1pc/[rrrruu]^{\Gamma_{0,\widetilde{z}_0}} & & & & \vdots & & & & \omega_{\infty}\ar@(dr,ur)[]_{\varphi_{\infty}}  \ar@/^2pc/[lllldddd]^{\Gamma_{\infty,\widetilde{z}_\infty}} \ar@/^1pc/[lllldd]^{\Gamma_{\infty,\wb}}\\
   & & & & \omega_1^{\left(\widetilde{z}_1'\right)}\ar@/_2pc/[d]_{\varphi_1^{\left(\widetilde{z}_1',\wb\right)}} & & & & \\
   & & & & \omega_1^{\left(\wb\right)} \ar@/_2pc/[dd]_{\varphi_1^{\left(\wb,\widetilde{z}_\infty\right)}} & & & & \\
   & & & & \vdots & & & & \\
   & & & & \omega_1^{\left(\widetilde{z}_{\infty}\right)} & & & & \\
}
$$

\subsection{Construction of isomorphisms of tensor functors}\label{sec:construction_isom}\leavevmode\par

Let $\left(A_0,A_1,A_\infty,\wMzero,\wM_\infty\right)\in\obj\left(\mathcal{C}_{rs}\right)$. In order to construct the paths $\Gamma_{0,\wa}$ and $\Gamma_{\infty,\wb}$ in the Galois groupoid $G$ of $\mathcal{C}_{rs}$, we need to evaluate the matrices $\wMzero$ and $\wM_\infty$. This is why we introduce the following definition. 

\begin{defin}
The \emph{singular locus of a matrix M} is 
$$S(M):=\pol(M)\cup \pol\left(M^{-1}\right)=\pol(M)\cup\zer\left(\det M\right).$$
\end{defin}

\subsubsection{The category \texorpdfstring{$\mathcal{C}_{E_0,E_\infty}$}{C(E0,E∞)}}

\begin{notat}
\label{nota:widetilde_Ei}
Let $E_0$ (respectively $E_\infty$) be a finite subset of $D(0,1)\setminus\lbrace 0\rbrace$ (respectively of $\mathbb{C}\setminus\overline{D}(0,1)$). In what follows, a typical choice for the sets $E_0$ and $E_\infty$ is 
$$
E_0=S(A)\cap\left(D(0,1)\setminus\lbrace 0\rbrace\right)\quad\mbox{and}\quad E_\infty=S(A)\cap \left(\mathbb{C}\setminus\overline{D}(0,1)\right)
$$
for a matrix $A\in \GL_n\left(\mathbb{C}(z)\right).$
For $i=0,\infty$, we define 
$$
E_i^{p^{\mathbb{Z}}}:=\bigcup\limits_{k\in\mathbb{Z}} E_i^{p^{k}}
$$
where $E_i^{p^k}:=\lbrace z\in\mathbb{C}\mid z^{p^{-k}}\in E_i\rbrace$ if $k\in\mathbb{Z}_{<0}$ and $E_i^{p^k}:=\lbrace z^{p^k}\mid z\in E_i\rbrace$ if $k\in\mathbb{Z}_{\geq 0}$.
We let
$$
\widetilde{E_i}:=\left\lbrace \left(z,\arg(z)+2\pi l\right)\left| z\in E_i^{p^{\mathbb{Z}}}, l\in\mathbb{Z}\right.\right\rbrace
$$
where for $z\in\mathbb{C}^*$, $\arg(z)$ denotes an argument of the complex number $z$.
 
We consider the full subcategory $\mathcal{C}_{E_0,E_\infty}$ of the category $\mathcal{C}_{rs}$ whose objects are the $\left(A_0,A_1,A_\infty, \wMzero, \wM_\infty\right)\in\obj\left(\mathcal{C}_{rs}\right)$ such that for $i=0,\infty$, 
$$S\left(\widetilde{M_i}\right)\subset \widetilde{E_i}.$$ 
\end{notat}

\begin{lem}
Every morphism in the category $\mathcal{C}_{E_0,E_\infty}$ has a kernel.
\end{lem}

\begin{proof}
We let $\left(S_0,\wS_1,S_\infty\right)$ denote a morphism in the category $\mathcal{C}_{E_0,E_\infty}$ from the object $\mathcal{X}:=~\left(A_0,A_1,A_\infty,\wMzero,\wM_\infty\right)$ of rank $n_1$ to the object $\mathcal{Y}:=\left(B_0,B_1,B_\infty,\wNzero,\wN_\infty\right)$ of rank $n_2$. By the equivalence between the categories $\mathcal{E}_{rs}$ and $\mathcal{C}_{rs}$ (see Corollary \ref{cor:equiv_global_cat}) and Lemma \ref{lem:kernels_in_E_rs}, there exist $$\mathcal{X}':=\left(A_0',A_1',A_\infty',\wMzero',\wM_\infty'\right)\in\obj\left(\mathcal{C}_{rs}\right)$$ of rank $n_1-r$ with $r=\rg(S_0)=\rg(\wS_1)=\rg(S_\infty)$ and $$\left(K_0,\widetilde{K_1},K_\infty\right)\in\ho\left(\mathcal{C}_{rs}\right)$$ such that
\begin{equation}
\left(\mathcal{X}',\left(K_0,\widetilde{K_1},K_\infty\right): \mathcal{X}'\rightarrow\mathcal{X} \right)
\end{equation}
is a kernel of the morphism $\left(S_0,\wS_1,S_\infty\right)$ in the category $\mathcal{C}_{rs}$ with $$\rg(K_0)=\rg\left(\widetilde{K_1}\right)=\rg(K_\infty)=n_1-r.$$ It remains to show that $\mathcal{X}'$ is an object of $\mathcal{C}_{E_0,E_\infty}$. Let us study the singular locus of $\wMzero'$. Given that $\left(K_0,\widetilde{K_1},K_\infty\right)\in\ho_{\mathcal{C}_{rs}}\left(\mathcal{X}',\mathcal{X}\right)$, we have $\widetilde{K_1}\wMzero'=\wMzero K_0$ that is 
\begin{equation}\label{eq:Mzero'_Mzero}
K_0\wMzero'^{-1}=\wMzero^{-1}\widetilde{K_1}.
\end{equation}
There exists a permutation matrix $T\in \GL_{n_1}\left(\mathbb{C}\right)$ such that
\begin{center}
$TK_0=\begin{pmatrix}
L_1\\ L_2
\end{pmatrix}$ with $L_1\in \GL_{n_1-r}\left(\mathbb{C}\right)$ and $L_2\in\mathcal{M}_{r,n_1-r}\left(\mathbb{C}\right)$
\end{center}
because $\rg(K_0)=n_1-r$.
The equality \eqref{eq:Mzero'_Mzero} implies
$$\wMzero'^{-1}=\left(L_1^{-1}\quad 0\right)T \wMzero^{-1} \widetilde{K_1}$$
so 
\begin{equation}\label{eq:poles_Mzero'}
\pol\left(\wMzero'^{-1}\right)\subset\widetilde{E_0}.
\end{equation}
In order to finish the study of the singular locus of $\wMzero'$, we show that 
\begin{equation}\label{eq:zeros_det_Mzero'}
\zer\left(\det\left(\wMzero'^{-1}\right)\right)\subset\widetilde{E_0}.
\end{equation}
Let $U\in \GL_{n_1}\left(\mathbb{C}\right)$ be a permutation matrix such that
\begin{center}
$U\widetilde{K_1}=\begin{pmatrix}
\widetilde{L_1}\\ \widetilde{L_2}
\end{pmatrix}$ with $\widetilde{L_1}\in \GL_{n_1-r}\left(\mathbb{C}\left[\wln,\wln^{-1}\right]\right)$ and $\widetilde{L_2}\in\mathcal{M}_{r,n_1-r}\left(\mathbb{C}\left[\wln,\wln^{-1}\right]\right)$.
\end{center}
By contradiction, assume that there exists 
$$\wa\in\zer\left(\det\left(\wMzero'^{-1}\right)\right)\cap\left(\sigz\setminus\widetilde{E_0}\right).$$
From \eqref{eq:poles_Mzero'}, we know that $\wa$ is not a pole of $\wMzero'^{-1}$ so we can consider $\wMzero'^{-1}\left(\wa\right)$ and there exists $v\in\mathbb{C}^{n_1-r}\setminus\lbrace 0\rbrace$ such that $$\wMzero'^{-1}\left(\wa\right)v=0\in\mathbb{C}^{n_1-r}.$$
The equality \eqref{eq:Mzero'_Mzero} implies
$$0=K_0\wMzero'^{-1}\left(\wa\right)v=\wMzero^{-1}\left(\wa\right)U^{-1}\begin{pmatrix}
\widetilde{L_1}\left(\wa\right)\\ \widetilde{L_2}\left(\wa\right)
\end{pmatrix}v$$
so $$\begin{pmatrix}
\widetilde{L_1}\left(\wa\right)\\ \widetilde{L_2}\left(\wa\right)
\end{pmatrix}v=0$$ 
because $\wMzero^{-1}\left(\wa\right)U^{-1}$ is an invertible matrix. In particular, $\wa\in\zer\left(\det\left(\widetilde{L_1}\right)\right)$. From the equality
$$\puip\left(\widetilde{K_1}\right)A'_1=A_1\widetilde{K_1},$$
we have 
$$\begin{pmatrix}
\widetilde{L_1}\left(\wa^p\right)\\ \widetilde{L_2}\left(\wa^p\right)
\end{pmatrix}A'_1v=0$$ so $\wa^p\in\zer\left(\det\left(\widetilde{L_1}\right)\right)$. By induction, it follows that for all $k\in\mathbb{N}$, $$\wa^{p^k}\in\zer\left(\det\left(\widetilde{L_1}\right)\right).$$ As $\wa\neq\widu$, $\det\left(\widetilde{L_1}\right)$ has an infinity of roots, it is a Laurent polynomial in $\wln$ so it is equal to $0$. This contradicts the fact that $\widetilde{L_1}$ is an invertible matrix.

Therefore, from \eqref{eq:poles_Mzero'} and \eqref{eq:zeros_det_Mzero'},
$$S\left(\wMzero'\right)\subset \widetilde{E_0}$$ and, similarly, we can show that 
$$S\left(\wM_\infty'\right)\subset \widetilde{E_\infty}.$$
\end{proof}

\begin{thm}
The category $\mathcal{C}_{E_0,E_\infty}$ is a neutral Tannakian category over $\mathbb{C}$.
\end{thm}

\begin{proof}
The category $\mathcal{C}_{E_0,E_\infty}$ is a subcategory of $\mathcal{C}_{rs}$ which is stable by tensor and abelian constructions. The only nontrivial point we need to prove is the existence of kernels in $\mathcal{C}_{E_0,E_\infty}$, which is given by the lemma hereabove.
\end{proof}

In what follows, we keep the same notations for the restrictions to $\mathcal{C}_{E_0,E_\infty}$ of the fibre functors $\omega_0, \omega_\infty, \omega_1^{(\wa)}, \wa\in\revc\setminus\left(\widetilde{E_0}\cup\widetilde{E_\infty}\cup \lbrace\widu\rbrace\right)$ defined in Proposition \ref{prop:fibre_functors}.  

\begin{prop}
\label{prop:catC_catCe0,einfty}
If $\mathcal{X}=\left(A_0,A_1,A_\infty,\wMzero,\wM_\infty\right)\in\obj(\mathcal{C}_{rs})$ then there exist two finite subsets $E_0$ and $E_\infty$ of $D(0,1)\setminus\lbrace 0 \rbrace$ and $\mathbb{C}\setminus\overline{D}(0,1)$ respectively, such that $\mathcal{X}\in\obj\left(\mathcal{C}_{E_0,E_\infty}\right)$. In particular, we have 
$$ E_0=S(A)\cap\left(D(0,1)\setminus\lbrace 0 \rbrace\right)\quad \mbox{and}\quad E_\infty=S(A)\cap \left(\mathbb{C}\setminus\overline{D}(0,1)\right)$$
for a matrix $A\in\GL_n\left(\mathbb{C}(z)\right)$.
\end{prop}

\begin{proof}
Using the notations of the category $\mathcal{C}_{rs}$, in Section \ref{sec:cat_connections}, the matrix 
\begin{equation}
\label{eq:relation_A_A0_W0}
A := \left\lbrace\begin{array}{ll}
\puip(W_0)A_0W_0^{-1} & \mbox{on}\quad D(0,1) \\
\puip(W_\infty)A_\infty W_\infty^{-1} & \mbox{on}\quad \mathbb{P}^1\left(\mathbb{C}\right)\setminus\overline{D}(0,1)
\end{array}\right.
\end{equation}
satisfies
\begin{equation}\label{eq:relation_A_A1_W1}
\pi^\star A=\puip\left(\wW_1\right)A_1\wW_1^{-1}.
\end{equation}
Therefore, from Lemma \ref{lem:merom_funct_inter}, $A\in \GL_n\left(\mathbb{C}(z)\right)$. From the equality $\wMzero=\wW_1^{-1}\pi^\star W_0$, we have $$S\left(\wMzero\right)\subset \left(S\left(\wW_1\right)\cup S\left(\pi^\star W_0\right)\right)\cap\sigz.$$
From the equality \eqref{eq:relation_A_A1_W1}, we have
$$\wW_1=\puip^{-1}\left(\pi^\star A\right)\cdots\puip^{-m}\left(\pi^\star A\right)\puip^{-m}\left(\wW_1\right)A_1^{-m}\quad\mbox{with}\quad m\in\mathbb{N}^\star.$$
For a $m\in\mathbb{N}^\star$ big enough and using that $\wW_1\in \GL_n\left(\mathcal{M}\left(\revc\right)\right)$, we obtain 
$$S\left(\wW_1\right)\cap\sigz\subset\widetilde{E_0}\quad \mbox{with}\quad E_0:=S(A)\cap\left(D(0,1)\setminus\lbrace 0 \rbrace\right),$$
see Notation \ref{nota:widetilde_Ei}. Similarly, from the equality \eqref{eq:relation_A_A0_W0}, we can show that $S\left(\pi^\star W_0\right)\subset\widetilde{E_0}.$ Therefore,
$$S\left(\wMzero\right)\subset\widetilde{E_0}$$ and $E_0$ is by construction a finite subset of $D(0,1)\setminus\lbrace 0\rbrace$.
The singular locus of $\wM_\infty$ satisfies the analogue property.
\end{proof}

\subsubsection{Construction of isomorphisms of tensor functors}\leavevmode\par

From Proposition \ref{prop:catC_catCe0,einfty}, we know that for every object $\mathcal{X}$ of $\mathcal{C}_{rs}$ there exist two finite sets $E_0\subset D(0,1)\setminus\lbrace 0 \rbrace$ and $E_\infty\subset \mathbb{C}\setminus\overline{D}(0,1)$ such that $\mathcal{X}$ is, in particular, an object of $\mathcal{C}_{E_0,E_\infty}$. From now on, we consider two finite subsets of $D(0,1)\setminus\lbrace 0 \rbrace$ and $\mathbb{C}\setminus\overline{D}(0,1)$, respectively denoted by $E_0$ and $E_\infty$.

\begin{prop}\label{prop:natural_transfo_gamma0}
Let $\wa\in\sigz\setminus\widetilde{E_0}$. The natural transformation
$$\Gamma_{0,\wa}: \left(A_0,A_1,A_\infty,\wMzero,\wM_\infty\right)\in\obj\left(\mathcal{C}_{E_0,E_\infty}\right)\leadsto \wMzero(\wa)$$ is an element of $\iso^\otimes \left(\omega_0,\omega_1^{\left(\wa\right)}\right)$.
\end{prop} 
  
\begin{proof}
To prove that $\Gamma_{0,\wa}\in\iso^\otimes \left(\omega_0,\omega_1^{\left(\wa\right)}\right)$, we have to verify the two following assumptions:
\begin{itemize}
\item Condition to be an isomorphism of functors. Let 
\begin{center}
$\mathcal{X}:=\left(A_0,A_1,A_\infty,\wMzero,\wM_\infty\right)$ and $\mathcal{Y}:=\left(B_0,B_1,B_\infty,\wNzero ,\wN_\infty\right)$ 
\end{center}
be two objects of $\mathcal{C}_{E_0,E_\infty}$ and $\left(S_0,\wS_1,S_\infty\right)\in\ho_{\mathcal{C}_{E_0,E_\infty}}\left(\mathcal{X},\mathcal{Y}\right)$. The following diagram is commutative because in the category of connection data $\wNzero S_0=\wS_1\wMzero$:
$$\xymatrix{\omega_0\left(\mathcal{X}\right) \ar[r]^{S_0} \ar[d]_{\wMzero\left(\wa\right)} & \omega_0\left(\mathcal{Y}\right) \ar[d]^{\wNzero\left(\wa\right)}\\
\omega_1^{(\wa)}\left(\mathcal{X}\right)\ar[r]_{\wS_1\left(\wa\right)} & \omega_1^{(\wa)}\left(\mathcal{Y}\right).}$$
This shows that it is a morphism of functors. Moreover, it is an isomorphism because the $\wMzero\left(\wa\right)$ are isomorphisms.
\item Condition of tensor compatibility. For all $\mathcal{X}, \mathcal{Y}\in\obj\left(\mathcal{C}_{E_0,E_\infty}\right)$,
$$\Gamma_{0,\wa}\left(\mathcal{X}\right)\otimes\Gamma_{0,\wa}\left(\mathcal{Y}\right)=\wMzero\left(\wa\right)\otimes\wNzero\left(\wa\right)=\left(\wMzero\otimes\wNzero\right)\left(\wa\right)=\Gamma_{0,\wa}\left(\mathcal{X}\otimes\mathcal{Y}\right).$$
\end{itemize}
\end{proof} 

Similarly, we have:
\begin{prop}\label{prop:natural_transfo_gamma0b}
Let $\wb\in\sigi\setminus\widetilde{E_\infty}$. The natural transformation $$\Gamma_{\infty,\wb}: \left(A_0,A_1,A_\infty,\wMzero,\wM_\infty\right)\in\obj\left(\mathcal{C}_{E_0,E_\infty}\right)\leadsto \wM_\infty(\wb)$$ is an element of $\iso^\otimes \left(\omega_\infty,\omega_1^{\left(\wb\right)}\right)$. 
\end{prop}

\begin{prop} We have:
\begin{itemize}
\item The natural transformation $$\left(A_0,A_1,A_\infty,\wMzero,\wM_\infty\right)\in\obj\left(\mathcal{C}_{E_0,E_\infty}\right)\leadsto A_0$$ is an element of $\aut^\otimes \left(\omega_0\right)$.

\item The natural transformation $$\left(A_0,A_1,A_\infty,\wMzero,\wM_\infty\right)\in\obj\left(\mathcal{C}_{E_0,E_\infty}\right)\leadsto A_\infty$$ is an element of $\aut^\otimes \left(\omega_\infty\right)$.

\item Let $\wa\in\revc\setminus\left(\widetilde{E_0}\cup\widetilde{E_\infty}\cup\lbrace\widu\rbrace\right)$. The natural transformation $$\left(A_0,A_1,A_\infty,\wMzero,\wM_\infty\right)\in\obj\left(\mathcal{C}_{E_0,E_\infty}\right)\leadsto A_1$$ is an element of $\iso^\otimes \left(\omega_1^{\left(\wa\right)},\omega_1^{\left(\wa^p\right)}\right)$.
\end{itemize}
\end{prop}

We will keep the same notations for the restriction of $G$ to the Galois groupoid of $\mathcal{C}_{E_0,E_\infty}$.

\begin{cor}\label{cor:eltms_of_galoisgroupoid}
The local Galois groupoids $G_0,G_1,G_\infty$, the $\Gamma_{0,\wa}$, $\wa\in\sigz\setminus\widetilde{E_0}$ (defined in Proposition \ref{prop:natural_transfo_gamma0}) and the $\Gamma_{\infty,\wb}$, $\wb\in\sigi\setminus\widetilde{E_\infty}$ (defined in Proposition \ref{prop:natural_transfo_gamma0b}) are elements of the Galois groupoid $G$ of $\mathcal{C}_{E_0,E_\infty}$. The subgroupoid generated by these elements is transitive.
\end{cor}

In the next subsection, we prove that they generate a Zariski-dense subgroupoid of the Galois groupoid $G$.

\subsection{Analogue of the density theorem of Schlesinger for Mahler equations}\label{sec:analogue_density-thm}
\subsubsection{Density criteria}\leavevmode\par
We use the following results (see \cite[Proposition 3.1(b), Remark 3.2(a)]{Deligne1982} or the proof of \cite[Proposition 2.8]{DelMil}):

\begin{thm}
Let $\mathcal{C}$ be a Tannakian category and $\omega$ a fibre functor for $\mathcal{C}$. Let $\left<\mathcal{X}\right>$ be the Tannakian subcategory generated by the object $\mathcal{X}$ of $\mathcal{C}$. Then the morphism 
$$ \begin{array}{rcl}
 \aut^\otimes\left({\omega}_{\mid\left<\mathcal{X}\right>}\right)& \rightarrow & \aut\left(\omega\left(\mathcal{X}\right)\right) \\
  g & \mapsto & g\left(\mathcal{X}\right)
    \end{array}$$
is injective. It identifies the Galois group of $\mathcal{X}$, denoted by $G\left(\mathcal{X}\right):=\aut^\otimes\left({\omega}_{\mid\left<\mathcal{X}\right>}\right)$, with an algebraic subgroup of $\aut\left(\omega\left(\mathcal{X}\right)\right)\simeq \GL_n\left(\mathbb{C}\right)$ ($n$ is the rank of $\mathcal{X}$ namely the dimension of the $\mathbb{C}$-vector space $\omega\left(\mathcal{X}\right)$).
\end{thm}

The following result is a density criterion, we keep the notations introduced in the previous theorem:

\begin{thm}\label{thm:density_criterion_grps}
Let $H$ be a subgroup of $G=\aut^\otimes(\omega)$. If $\mathcal{X}$ is an object of $\mathcal{C}$, we denote by $H\left(\mathcal{X}\right)$ the image of $H$ in $G\left(\mathcal{X}\right)$.

We assume that for all objects $\mathcal{Y}$ of $\mathcal{C}$, for all $y\in\omega\left(\mathcal{Y}\right)$, if the line $\mathbb{C}y$ is stable by $H\left(\mathcal{Y}\right)$ then it is also stable by $G\left(\mathcal{Y}\right)$.

Then, for all objects $\mathcal{X}$ of $\mathcal{C}$, $H\left(\mathcal{X}\right)$ is Zariski-dense in $G\left(\mathcal{X}\right)$.
\end{thm}

Herebelow, it is the analogue of the previous density criterion for groupoids. 
\begin{thm}\label{thm:density_criterion_groupoid}
We consider a transitive groupoid $G$ whose objects are fibre functors $\omega_i$ for a Tannakian category $\mathcal{C}$ and whose morphisms from the object $\omega_i$ to the object $\omega_j$ are the elements of $G\left(\omega_i,\omega_j\right):=\iso^\otimes\left(\omega_i,\omega_j\right)$. Let $H$ be a subgroupoid of $G$ which is transitive.

We assume that for all objects $\mathcal{Y}$ of $\mathcal{C}$, for all $y_i\in\omega_i\left(\mathcal{Y}\right)$, if the lines $\mathbb{C}y_i$ are globally stable by $H\left(\mathcal{Y}\right)$ then they are also globally stable by $G\left(\mathcal{Y}\right)$.

Then, for all objects $\mathcal{X}$ of $\mathcal{C}$, $H\left(\mathcal{X}\right)$ is Zariski-dense in $G\left(\mathcal{X}\right)$.
\end{thm}

In a Picard-Vessiot theory, the main argument to prove (analogues of) the density theorem of Schlesinger for difference/differential Galois groups is the Galois correspondence. In the theory of Tannakian categories, it corresponds to Theorem \ref{thm:density_criterion_groupoid}.

\subsubsection{Density theorem in the regular singular case}\leavevmode\par

We consider the following regular singular Mahler system at $0,1$ and $\infty$ 
$$\puip(Y)=AY$$  
with $A(z)\in \GL_n\left(\mathbb{C}(z)\right)$.
    
We recall that we consider the full subcategory $\mathcal{C}_{E_0,E_\infty}$ of $\mathcal{C}_{rs}$.

\begin{thm}\label{thm:density_regsing}
The local Galois groupoids $G_0, G_1, G_\infty$ and the Galoisian isomorphisms
$\Gamma_{0,\widetilde{z_0}}$, $\Gamma_{\infty,\widetilde{z_\infty}}$ for all $\widetilde{z_0}\in\sigz\setminus\widetilde{E_0}$, $\widetilde{z_\infty}\in\sigi\setminus\widetilde{E_\infty}$ generate a Zariski-dense subgroupoid $H$ of the Galois groupoid $G$ of $\mathcal{C}_{E_0,E_\infty}$. 
\end{thm}
 
\begin{proof}
To prove this theorem, we use the density criterion given in Theorem \ref{thm:density_criterion_groupoid}. Let 
$$\mathcal{X}=\left(A_0,A_1,A_\infty, \wMzero, \wM_\infty\right)$$
be an object of $\mathcal{C}_{E_0,E_\infty}$. For each object $\omega_0$, $\omega_1^{(\wc)}$, $\omega_\infty$, $\wc\in\revc\setminus\left(\widetilde{E_0}\cup\widetilde{E_\infty}\cup \lbrace\widu\rbrace\right)$ of $G$, we choose a line $D_i\subset\omega_i\left(\mathcal{X}\right)=\mathbb{C}^n$, $i=0,\infty$ and $\widetilde{D_1}^{(\wc)}\subset\omega_1^{(\wc)}\left(\mathcal{X}\right)=\mathbb{C}^n$. We assume that this family of lines is globally stable by $H$. We have to show that this family of lines is also globally stable by $G$.

We denote by $\rho : G\rightarrow \GL\left(\omega\left(\mathcal{X}\right)\right)$ the representation of $G$ corresponding to $\mathcal{X}$ by Tannaka duality.

The restriction $\rho_{\mid G_1} : G_1\rightarrow \GL\left(\omega\left(\mathcal{X}\right)\right)$ is the representation corresponding to $A_1$ by Tannaka duality for the category $\mathcal{P}^{(1)}$. By hypothesis, the lines $\widetilde{D_1}^{(\wc)}$ are globally stable by the action of $G_1$ so it induces a subrepresentation of rank $1$ of $\rho_{\mid G_1}$. Therefore, by Tannaka duality, it comes from a subobject of rank $1$ of $A_1$ in $\mathcal{P}^{(1)}$. This subobject is an object $a_1$ of $\mathcal{P}^{(1)}$ of rank $1$ and a monomorphism $\widetilde{d_1}:a_1\rightarrow A_1$ of $\mathcal{P}^{(1)}$. 
In short, $a_1\in\mathbb{C}^\star$ and $\widetilde{d_1}$ is a column matrix whose entries are Laurent polynomials in $\wln$ such that
\begin{equation}\label{eq:A1_and_a1}
\puip\left(\widetilde{d_1}\right)a_1=A_1\widetilde{d_1}
\end{equation}
and $$\omega_1^{(\wc)}\left(\widetilde{d_1}\right)\mathbb{C}=\widetilde{d_1}\left(\wc\right)\mathbb{C}=\widetilde{D_1}^{(\wc)}.$$

The restriction $\rho_{\mid G_0} : G_0\rightarrow \GL\left(\omega\left(\mathcal{X}\right)\right)$ is a representation which corresponds to $A_0$ by Tannaka duality for the category $\mathcal{P}^{(0)}$. By hypothesis, the line $D_0$ is stable by the action of $G_0$ so it induces a subrepresentation of rank $1$ of $\rho_{\mid G_0}$. Therefore, it comes from a subobject of rank $1$ of $A_0$ in $\mathcal{P}^{(0)}$ namely an object $a_0$ of $\mathcal{P}^{(0)}$ of rank $1$ and a monomorphism $d_0 : a_0\rightarrow A_0$ of $\mathcal{P}^{(0)}$. 
In short, $a_0\in\mathbb{C}^\star$ and $d_0\in\mathbb{C}^n$ is such that
\begin{equation}\label{eq:A0_and_a0}
d_0a_0=A_0d_0
\end{equation}
and $$\mathbb{C}d_0=D_0.$$

Similarly, for the category $\mathcal{P}^{(\infty)}$ and the groupoid $G_\infty$, we have that
\begin{equation}\label{eq:Ainfty_and_ainfty}
d_\infty a_\infty=A_\infty d_\infty
\end{equation}
with $a_\infty\in\mathbb{C}^\star$, $d_\infty\in\mathbb{C}^n$ and $$\mathbb{C}d_\infty=D_\infty.$$

For all $\wa\in\sigz\setminus\widetilde{E_0}$, the stability under $\Gamma_{0,\wa}$ of the family of lines means that $\Gamma_{0,\wa}\left(\mathcal{X}\right):~ \omega_0\left(\mathcal{X}\right)\rightarrow\omega_1^{(\wa)}\left(\mathcal{X}\right)$ satisfies $\Gamma_{0,\wa}\left(\mathcal{X}\right)D_0=\widetilde{D_1}^{\left(\wa\right)}$, that is,
$$\wMzero\left(\wa\right)D_0=\widetilde{D_1}^{\left(\wa\right)}.$$
Consequently, for all $\wa\in\sigz\setminus\widetilde{E_0}$, there exists $\wmzero\left(\wa\right)\in\mathbb{C}$ such that $\wMzero\left(\wa\right)d_0=\wmzero\left(\wa\right)\widetilde{d_1}\left(\wa\right)$. Therefore, we obtain a function $\wmzero:\sigz\rightarrow\mathbb{C}$ such that
\begin{equation}\label{eq:Mzero_mzero}
\wMzero d_0=\wmzero\widetilde{d_1}.
\end{equation}
Similarly, for all $\wb\in\sigi\setminus\widetilde{E_\infty}$, $D_\infty=\wM_\infty\left(\wb\right)^{-1}D_1^{\left(\wb\right)}$ and there exists $\wm_\infty:\sigi\rightarrow\mathbb{C}$ such that
\begin{equation}\label{eq:Minfty_minfty}
\wM_\infty d_\infty=\wm_\infty\widetilde{d_1}.
\end{equation}
Thus, we have:
\begin{itemize}
\item $\mathcal{X}':=\left(a_0,a_1,a_\infty,\wmzero,\wm_\infty\right)$ is an object of $\mathcal{C}_{E_0,E_\infty}$. Indeed, the equality \eqref{eq:Mzero_mzero} implies that $\wmzero$ is a meromorphic function on $\sigz$. 
From the equality
$$\puip\left(\wMzero\right)=A_1\wMzero A_0^{-1},$$
satisfied in the category $\mathcal{C}_{E_0,E_\infty}$  (see Section \ref{sec:cat_connections}), we have 
$$\puip\left(\wMzero\right)d_0=A_1\wMzero A_0^{-1}d_0$$
and the equalities \eqref{eq:A1_and_a1}, \eqref{eq:A0_and_a0} imply
$$\puip\left(\wmzero\right)=a_1\wmzero a_0^{-1}.$$
The same argument shows that $\wm_\infty$ is a meromorphic function on $\sigi$ and satisfies
$$\puip\left(\wm_\infty\right)=a_1\wm_\infty a_\infty^{-1}.$$
Moreover, the singular loci are such that $S\left(\widetilde{m_i}\right)\subset\widetilde{E_i}$, $i=0,\infty$, for more details see the following remark.

\begin{rem}
Let us look at the singular locus of $\wmzero$. The equality \eqref{eq:Mzero_mzero},
$$\wmzero^{-1}d_0=\wMzero^{-1}\widetilde{d_1},$$
implies that 
$$\pol\left(\wmzero^{-1}\right)\subset\left(\pol\left(\wMzero^{-1}\right)\cup\pol\left(\widetilde{d_1}\right)\right)\cap\sigz\subset\widetilde{E_0}.$$
Also, $\zer\left(\wmzero^{-1}\right)\subset\widetilde{E_0}$ because if $\widetilde{\alpha}\in\zer\left(\wmzero^{-1}\right)\cap\left(\sigz\setminus\widetilde{E_0}\right)$ then $\wMzero\left(\widetilde{\alpha}\right)^{-1}$ is well-defined and invertible but $0=\wMzero\left(\widetilde{\alpha}\right)^{-1}\widetilde{d_1}\left(\widetilde{\alpha}\right)$. This implies that $\widetilde{d_1}\left(\widetilde{\alpha}\right)= 0$ and $\vect_\mathbb{C}\widetilde{d_1}\left(\widetilde{\alpha}\right)=D_1^{\left(\widetilde{\alpha}\right)}= 0$, which is absurd. Therefore,
$$S\left(\wmzero\right)\subset\widetilde{E_0}.$$
Similarly,
$$S\left(\wm_\infty\right)\subset\widetilde{E_\infty}.$$
\end{rem}

In order to prove that $\mathcal{X}'$ is an object of $\mathcal{C}_{E_0,E_\infty}$, it remains to prove that there exist nonzero functions $\ww_1\in\mathcal{M}\left(\revc\right)$, $w_0\in\mathcal{M}\left(D(0,1)\right)$ and $w_\infty\in\mathcal{M}\left(\mathbb{P}^1\left(\mathbb{C}\right)\setminus\overline{D}(0,1)\right)$ such that
$$\left\lbrace
\begin{array}{l}
\ww_1\wmzero=\pi^\star w_0 \\
\ww_1\wm_\infty=\pi^\star w_\infty .
\end{array}
    \right.$$ 
Since $\left(A_0,A_1,A_\infty,\wMzero ,\wM_\infty\right)$ is an object of $\mathcal{C}_{E_0,E_\infty}$, there exist
\begin{center} 
 $\wW_1\in \GL_n\left(\mathcal{M}\left(\revc\right)\right)$, $W_0\in \GL_n\left(\mathcal{M}\left(D(0,1)\right)\right)$, $W_\infty\in \GL_n\left(\mathcal{M}\left(\mathbb{P}^1\left(\mathbb{C}\right)\setminus\overline{D}(0,1)\right)\right)$
\end{center}
such that $\wW_1\wMzero=\pi^\star W_0$ and $\wW_1\wM_\infty=\pi^\star W_\infty$. Therefore,
$$\left\lbrace
\begin{array}{l}
\pi^\star W_0 d_0=\wW_1\wMzero d_0=\wmzero\wW_1\widetilde{d_1}\quad\textnormal{on}\quad\sigz \\
\pi^\star W_\infty d_\infty=\wW_1\wM_\infty d_\infty=\wm_\infty\wW_1\widetilde{d_1}\quad\textnormal{on}\quad\sigi.
\end{array}
    \right.$$
A nonzero entry of the column matrix $\wW_1\widetilde{d_1}$ satisfies the conditions to be $\ww_1$. 
\item $m=\left(d_0,\widetilde{d_1},d_\infty\right):\mathcal{X}'\leadsto\mathcal{X}$ is a morphism of $\mathcal{C}_{E_0,E_\infty}$. This is due to the fact that $d_0, d_\infty\in\mathbb{C}^n$, the entries of $\widetilde{d_1}$ are Laurent polynomials in $\wln$ and
$$\left\lbrace
\begin{array}{l}
\widetilde{d_1}\wmzero=\wMzero d_0\\
\widetilde{d_1}\wm_\infty=\wM_\infty d_\infty .
\end{array}
    \right.$$
\end{itemize} 
The family of lines, fixed by the subgroupoid $H$, are globally stable by the groupoid $G$ because:
\begin{itemize}
\item for all $g_i\in\aut^\otimes\left({\omega_i}_{\mid\left<\mathcal{X}\right>}\right)$, $i=0,\infty$, the following diagram is commutative
$$\xymatrix{\omega_i\left(\mathcal{X}'\right) \ar[r]^{g_i\left(\mathcal{X}'\right)} \ar[d]_{\omega_i(m)=d_i} & \omega_i\left(\mathcal{X}'\right) \ar[d]^{\omega_i(m)=d_i}\\
\omega_i\left(\mathcal{X}\right)\ar[r]_{g_i\left(\mathcal{X}\right)} & \omega_i\left(\mathcal{X}\right)}$$
that is, $g_i\left(\mathcal{X}\right)d_i=d_i\underbrace{g_i\left(\mathcal{X}'\right)}_{\in\mathbb{C}^\star}=g_i\left(\mathcal{X}'\right)d_i$ so $g_i\left(\mathcal{X}\right)D_i=D_i$.
\item for all $g_1\in\iso^\otimes\left({\omega_1^{\left(\widetilde{c_1}\right)}}_{\mid\left<\mathcal{X}\right>},{\omega_1^{\left(\widetilde{c_2}\right)}}_{\mid\left<\mathcal{X}\right>}\right)$, the following diagram is commutative
$$\xymatrix{\omega_1^{\left(\widetilde{c_1}\right)}\left(\mathcal{X}'\right) \ar[r]^{g_1\left(\mathcal{X}'\right)} \ar[d]_{\omega_1^{\left(\widetilde{c_1}\right)}(m)=\widetilde{d_1}\left(\widetilde{c_1}\right)} & \omega_1^{\left(\widetilde{c_2}\right)}\left(\mathcal{X}'\right) \ar[d]^{\omega_1^{\left(\widetilde{c_2}\right)}(m)=\widetilde{d_1}\left(\widetilde{c_2}\right)}\\
\omega_1^{\left(\widetilde{c_1}\right)}\left(\mathcal{X}\right)\ar[r]_{g_1\left(\mathcal{X}\right)} & \omega_1^{\left(\widetilde{c_2}\right)}\left(\mathcal{X}\right)}$$
so $g_1\left(\mathcal{X}\right)\widetilde{D_1}^{\left(\widetilde{c_1}\right)}=\widetilde{D_1}^{\left(\widetilde{c_2}\right)}$.
\item similarly, for all $h_0\in\iso^\otimes\left({\omega_0}_{\mid\left<\mathcal{X}\right>},{\omega_1^{\left(\wa\right)}}_{\mid\left<\mathcal{X}\right>}\right)$ and $h_\infty\in\iso^\otimes\left({\omega_\infty}_{\mid\left<\mathcal{X}\right>},{\omega_1^{\left(\wb\right)}}_{\mid\left<\mathcal{X}\right>}\right)$, we can prove that $h_0\left(\mathcal{X}\right)D_0=\widetilde{D_1}^{\left(\wa\right)}$ and $h_\infty\left(\mathcal{X}\right)D_\infty=\widetilde{D_1}^{\left(\wb\right)}$.
\end{itemize}
\end{proof}

\subsubsection{Density theorem in the regular case}
\begin{defin}
Let $i\in\lbrace 0,1,\infty\rbrace$. A Mahler system $\puip(Y)=AY$, $A\in\GL_n\left(\mathbb{C}(z)\right)$, is regular at $i$ if it is strictly Fuchsian at $i$ and $A(i)=I_n$.

We say that $A\in\obj\left(\mathcal{E}_{rs}\right)$ is a \emph{regular object of the category $\mathcal{E}_{rs}$} if the system $\puip(Y)=AY$ is meromorphically equivalent at $0$, $1$ and $\infty$ to a regular system at $0$, $1$ and $\infty$ respectively.

We say that $\left(A_0,A_1,A_\infty,\wMzero,\wM_\infty\right)\in\obj\left(\mathcal{C}_{rs}\right)$ of rank $n\in\mathbb{N}^\star$ is a \emph{regular object of the category $\mathcal{C}_{rs}$} if $A_0=A_1=A_\infty=I_n$.
\end{defin}
 
We denote by $\mathcal{E}_r$ the full subcategory of $\mathcal{E}_{rs}$ whose objects are the regular objects of $\mathcal{E}_{rs}$.

We denote by $\mathcal{C}_r$ the full subcategory of $\mathcal{C}_{rs}$ whose objects are the regular objects of $\mathcal{C}_{rs}$.

These categories are Tannakian subcategories of $\mathcal{E}_{rs}$ and $\mathcal{C}_{rs}$, respectively. We can check that the categories $\mathcal{E}_r$ and $\mathcal{C}_r$ are equivalent.

We denote by $\mathcal{C}_{r,E_0,E_\infty}$ the full subcategory of $\mathcal{C}_{E_0,E_\infty}$ whose objects are the regular objects. It is a neutral Tannakian category over $\mathbb{C}$. From Theorem \ref{thm:density_regsing}, we obtain:

\begin{cor}
Let $\mathcal{X}=\left(I_n,I_n,I_n,\wMzero , \wM_\infty\right)$ be an object of the category $\mathcal{C}_{r,E_0,E_\infty}$.  The subgroup $H$ of $G=\aut^\otimes\left(\omega_1\right)$ containing all
$$
\begin{array}{l}
\mathcal{X}\leadsto \wMzero\left(\wa_1\right)\wMzero\left(\wa_2\right)^{-1}\\
\mathcal{X}\leadsto \wM_\infty\left(\wb_1\right)\wM_\infty\left(\wb_2\right)^{-1}
\end{array}
$$
with $\wa_1, \wa_2\in\sigz\setminus\widetilde{E_0}$ and $\wb_1,\wb_2\in\sigi\setminus\widetilde{E_\infty}$ is Zariski-dense in $G$. 
 \end{cor}
 
 \begin{proof}
It follows from Theorem \ref{thm:density_regsing} and the fact that in this case the local Galois groupoids $G_0,G_1,G_\infty$ are trivial because $A_0=A_1=A_\infty=I_n$.
 \end{proof}
 
\begin{appendix}
\addtocontents{toc}{\setcounter{tocdepth}{1}}

\section{Tannakian Categories}
 \label{sec:appendix}
 
\subsection{Abelian categories}\leavevmode\par
For more details, the reader is referred to \cite{Sch72}. 

\begin{defin}
%[see Definitions 7.2.5 and 8.2.5]
Let $\mathcal{C}$ be a category with zero morphisms, denoted by $0$. Let $X,Y\in\obj\left(\mathcal{C}\right)$ and $f\in\ho_{\mathcal{C}}\left(X,Y\right)$.  A \emph{kernel} of $f$ is a pair $\left(K, k: K\rightarrow X\right)$ where $K\in\obj\left(\mathcal{C}\right)$ and $k\in\ho_{\mathcal{C}}\left(K,X\right)$ such that:
\begin{itemize}
\item $fk=0$, 
\item for all $K'\in\obj\left(\mathcal{C}\right)$ and all morphisms $k':K'\rightarrow X$, if $f k'=0$ then there exists a unique morphism $u:K'\rightarrow K$ such that $k'=k u$. 
\end{itemize}
A \emph{cokernel} of $f$ is a pair $\left(Q, q: Y\rightarrow Q\right)$ where $Q\in\obj\left(\mathcal{C}\right)$ and $q\in\ho_{\mathcal{C}}\left(Y,Q\right)$ such that: 
\begin{itemize}
\item $q f=0$ ;
\item for all $Q'\in\obj\left(\mathcal{C}\right)$ and all morphisms $q':Y\rightarrow Q'$, if $q' f=0$ then there exists a unique morphism $v:Q\rightarrow Q'$ such that $q'=v q$. 
\end{itemize}  
\end{defin}

\begin{defin}
An \emph{abelian category} is a category which satisfies the following conditions:
\begin{description}\itemsep=0pt
\item[$A_0$] There is a zero object.
\item[$A_1$] There are finite products.
\item[$B_1$] There are finite coproducts.
\item[$A_2$] Every morphism has a kernel.
\item[$B_2$] Every morphism has a cokernel.
\item[$A_3$] Every monomorphism is a kernel.
\item[$B_3$] Every epimorphism is a cokernel.
\end{description}
\end{defin}

\begin{rem}
The condition $B_1$ (or $A_1$) can be dropped (see \cite[Proposition 12.5.1]{Sch72}).
\end{rem}

For further details on what follows, the reader is referred to \cite{DelMil}. 
\subsection{Rigid tensor categories}

\begin{defin}
A tensor category $\left(\mathcal{C},\otimes\right)$, see \cite[Definition 1.1]{DelMil}, is \emph{rigid} if:
\begin{itemize}
\item internal Hom exists for all $X,Y\in\obj\left(\mathcal{C}\right)$, it is denoted by $\underline{\text{Hom}}\left(X,Y\right)$,
\item the morphisms 
$$\underline{\text{Hom}}\left(X_1,Y_1\right)\otimes \underline{\text{Hom}}\left(X_2,Y_2\right) \rightarrow \underline{\text{Hom}}\left(X_1\otimes X_2, Y_1\otimes Y_2\right)$$ are isomorphisms for all $X_1,X_2,Y_1,Y_2\in\obj\left(\mathcal{C}\right)$, 
\item all objects of $\mathcal{C}$ are reflexive.
\end{itemize}
\end{defin}

\subsection{Tannakian Categories}\leavevmode\par
 
From \cite[Definition 1.15]{DelMil} and \cite[Proposition 1.16]{DelMil}, we have the following definition.
\begin{defin}
A \emph{rigid abelian tensor category} is a rigid tensor category $\left(\mathcal{C},\otimes\right)$ such that $\mathcal{C}$ is an abelian category.
\end{defin}

Let $k$ be a field and $K$ be a field extension of $k$.
\begin{defin}
A rigid abelian tensor category $\mathcal{C}$, whose ring of endomorphisms of the unit object is $k$, is a \emph{Tannakian category over $k$} if it admits an exact faithful $k$-linear tensor functor 
$
\mathcal{C}\rightarrow\vect_{K}^f.
$
Any such functor is said to be a \emph{fibre functor} for $\mathcal{C}$ with values in $K$. 

A Tannakian category over $k$ is a \emph{neutral Tannakian category over $k$} if it admits a fibre functor with values in $k$.
\end{defin}

\begin{ex}
Let $G$ be an affine group scheme over $k$. The category $\mbox{Rep}_k\left(G\right)$ of finite dimensional representations of $G$ over $k$ is a neutral Tannakian category over $k$.
\end{ex}
 
Let $\left(\mathcal{C},\otimes\right)$ be a tensor category, $R$ be a $k$-algebra and $\mbox{Mod}_R$ be the category of finitely generated $R$-modules. We let $\Phi_R:\vect_{k}^f\rightarrow \mbox{Mod}_R$, $V\in\obj\left(\vect_{k}^f\right)\leadsto V\otimes_k R$ be the canonical tensor functor.
\begin{defin}
Let $F,G:\mathcal{C}\rightarrow\vect_{k}^f$ be two tensor functors. The functor $\underline{\iso}^\otimes\left(F,G\right)$ is the functor of $k$-algebra such that for all $k$-algebra $R$, 
$$\underline{\iso}^\otimes\left(F,G\right)(R):=\iso^\otimes\left(\Phi_R\circ F,\Phi_R\circ G\right).$$
We write $$\underline{\aut}^\otimes\left(F\right):=\underline{\iso}^\otimes\left(F,F\right).$$
\end{defin}

The main theorem of the theory of Tannakian categories is the following (see \cite[Theorem 2.11]{DelMil}).

\begin{thm}
Let $\mathcal{C}$ be a neutral Tannakian category over $k$ and $\omega: \mathcal{C}\rightarrow \vect_{k}^f$ be a fibre functor. Then:
\begin{itemize}
\item the functor $\underline{\aut}^\otimes\left(\omega\right)$ is represented by an affine group scheme $G$;
\item the functor $\mathcal{C}\rightarrow\mbox{Rep}_k\left(G\right)$ defined by $\omega$ is an equivalence of categories.
\end{itemize}
\end{thm}
\end{appendix}

\vspace{0.5cm}

\textbf{Acknowledgements.} This work was performed within the framework of the LABEX MILYON (ANR-10-LABX-0070) of Universit\'e de Lyon, within the program ``Investissements d'Avenir'' (ANR-11-IDEX- 0007) operated by the French National Research Agency (ANR).

\bibliographystyle{alpha}
\bibliography{densitythm}

\begin{thebibliography}{CDDM18}

\bibitem[AB17]{AdamBell17}
Boris Adamczewski and Jason~P. Bell.
\newblock A problem about {M}ahler functions.
\newblock {\em Ann. Sc. Norm. Super. Pisa Cl. Sci.}, 17:1301--1355, 2017.

\bibitem[ADH21]{ADH21}
Boris Adamczewski, Thomas Dreyfus, and Charlotte Hardouin.
\newblock Hypertranscendence and linear difference equations.
\newblock {\em J. Amer. Math. Soc.}, 34:475--503, 2021.

\bibitem[AF17]{AF17}
Boris Adamczewski and Colin Faverjon.
\newblock M{\'e}thode de {M}ahler : relations lin{\'e}aires, transcendance et
  applications aux nombres automatiques.
\newblock {\em Proc. Lond. Math. Soc.}, 115:55--90, 2017.

\bibitem[AF18]{AF18}
Boris Adamczewski and Colin Faverjon.
\newblock M{\'e}thode de {M}ahler, transcendance et relations lin{\'e}aires :
  aspects effectifs.
\newblock {\em J. Th{\'e}or. Nombres Bordeaux}, 30:557--573, 2018.

\bibitem[AS92]{AS92}
Jean-Paul Allouche and Jeffrey Shallit.
\newblock The ring of k-regular sequences.
\newblock {\em Theoret. Comput. Sci.}, 98(2):163--197, 1992.

\bibitem[BCR13]{BCR13}
Jason~P. Bell, Michael Coons, and Eric Rowland.
\newblock The rational-transcendental dichotomy of {M}ahler functions.
\newblock {\em J. Integer Seq.}, 16, 2013.

\bibitem[BCZ15]{BCZ15}
Richard~P. Brent, Michael Coons, and Wadim Zudilin.
\newblock {Algebraic Independence of {M}ahler Functions via Radial
  Asymptotics}.
\newblock {\em Int. Math. Res. Not. IMRN}, 2016(2):571--603, 2015.

\bibitem[Bec94]{Bec}
Paul-Georg Becker.
\newblock {$k$}-regular power series and {M}ahler-type functional equations.
\newblock {\em J. Number Theory}, 49(3):269--286, 1994.

\bibitem[CDDM18]{CDDM18}
Fr{\'e}d{\'e}ric Chyzak, Thomas Dreyfus, Philippe Dumas, and Marc Mezzarobba.
\newblock Computing solutions of linear {M}ahler equations.
\newblock {\em Math. Comp.}, 87:2977--3021, 2018.

\bibitem[Cob68]{Co68}
Alan Cobham.
\newblock On the {H}artmanis-{S}tearns problem for a class of tag machines.
\newblock In {\em 9th Annual Symposium on Switching and Automata Theory (swat
  1968)}, pages 51--60, 1968.

\bibitem[Del82]{Deligne1982}
Pierre Deligne.
\newblock Hodge cycles on abelian varieties.
\newblock In {\em Hodge Cycles, Motives, and Shimura Varieties}, volume 900 of
  {\em Lecture Notes in Math.}, pages 9--100. Springer Berlin Heidelberg,
  Berlin, Heidelberg, 1982.

\bibitem[DHR18]{DHR18}
Thomas Dreyfus, Charlotte Hardouin, and Julien Roques.
\newblock Hypertranscendence of solutions of {M}ahler equations.
\newblock {\em J. Eur. Math. Soc.}, 20(9):2209--2238, 2018.

\bibitem[DM82]{DelMil}
Pierre Deligne and James~S. Milne.
\newblock Tannakian categories.
\newblock In {\em Hodge Cycles, Motives, and Shimura Varieties}, volume 900 of
  {\em Lecture Notes in Math.}, pages 101--228. Springer Berlin Heidelberg,
  Berlin, Heidelberg, 1982.

\bibitem[Eti95]{etingof}
Pavel~I. Etingof.
\newblock Galois groups and connection matrices of q-difference equations.
\newblock {\em Electron. Res. Announc. Amer. Math. Soc.}, 1:1--9, 1995.

\bibitem[Fer19]{Fer19}
Gwladys Fernandes.
\newblock Regular extensions and algebraic relations between values of {M}ahler
  functions in positive characteristic.
\newblock {\em Trans. Amer. Math. Soc.}, 372:7111--7140, 2019.

\bibitem[HS99]{HS99}
Peter~A. Hendriks and Michael~F. Singer.
\newblock Solving difference equations in finite terms.
\newblock {\em J. Symbolic Comput.}, 27(3):239--259, 1999.

\bibitem[Lox84]{Lox84}
John~H. Loxton.
\newblock A method of {M}ahler in transcendence theory and some of its
  applications.
\newblock {\em Bull. Aust. Math. Soc.}, 29(1):127–136, 1984.

\bibitem[LvdP77]{LoxPoort77}
John~H. Loxton and Alfred~J. van~der Poorten.
\newblock Transcendence and algebraic independence by a method of {M}ahler.
\newblock In {\em Transcendence theory: advances and applications ({P}roc.
  {C}onf., {U}niv. {C}ambridge, {C}ambridge, 1976)}, pages 211--226, 1977.

\bibitem[Mah29]{Mah1}
Kurt Mahler.
\newblock Arithmetische {E}igenschaften der {L}{\"o}sungen einer {K}lasse von
  {F}unktionalgleichungen.
\newblock {\em Math. Ann.}, 101(1):342--366, 1929.

\bibitem[Mah30a]{Mah2}
Kurt Mahler.
\newblock Arithmetische {E}igenschaften einer {K}lasse
  transzendental-transzendenter {F}unktionen.
\newblock {\em Math. Z.}, 32(1):545--585, 1930.

\bibitem[Mah30b]{Mah3}
Kurt Mahler.
\newblock {\"U}ber das {V}erschwinden von {P}otenzreihen mehrerer
  {V}er{\"a}nderlichen in speziellen {P}unktfolgen.
\newblock {\em Math. Ann.}, 103(1):573--587, 1930.

\bibitem[MF80]{MenFr80}
Michel Mend{\`e}s~France.
\newblock Nombres alg{\'e}briques et th{\'e}orie des automates.
\newblock {\em Enseign. Math.}, 26:193--199, 1980.

\bibitem[Nis96]{Nish96}
Kumiko Nishioka.
\newblock {\em Mahler Functions and Transcendence}, volume 1631 of {\em Lecture
  Notes in Math.}
\newblock Springer-Verlag Berlin Heidelberg, 1996.

\bibitem[Pel20]{Pel11}
Federico Pellarin.
\newblock An introduction to {M}ahler’s method for transcendence and
  algebraic independence.
\newblock In {\em $t$-Motives: Hodge Structures, Transcendence and Other
  Motivic Aspects}, pages 297--349. EMS Ser. Congr. Rep., 2020.

\bibitem[Phi15]{Phi15}
Patrice Philippon.
\newblock Groupes de {G}alois et nombres automatiques.
\newblock {\em J. Lond. Math. Soc. (2)}, 92(3):596--614, 2015.

\bibitem[Ran92]{Ran92}
Bernard Rand{\'e}.
\newblock {\em {{\'E}quations fonctionnelles de {M}ahler et applications aux
  suites p-r{\'e}guli{\`e}res}}.
\newblock {P}h{D} thesis, {Universit{\'e} Bordeaux 1}, 1992.

\bibitem[Roq18]{Roq18}
Julien Roques.
\newblock On the algebraic relations between {M}ahler functions.
\newblock {\em Trans. Amer. Math. Soc.}, 370(1):321--355, 2018.

\bibitem[Sau00]{Sau00}
Jacques Sauloy.
\newblock Syst{\`e}mes aux {$q$}-diff{\'e}rences singuliers r{\'e}guliers :
  classification, matrice de connexion et monodromie.
\newblock {\em Ann. Inst. Fourier (Grenoble)}, 50(4):1021--1071, 2000.

\bibitem[Sau03]{Sau03}
Jacques Sauloy.
\newblock Galois theory of {F}uchsian {$q$}-difference equations.
\newblock {\em Ann. Sci. \'{E}c. Norm. Sup\'{e}r. (4)}, Ser. 4, 36(6):925--968,
  2003.

\bibitem[Sch95]{Sch1895}
Ludwig Schlesinger.
\newblock {\em Handbuch der Theorie der Linearen Differentialgleichungen}.
\newblock Teubner, 1895.

\bibitem[Sch72]{Sch72}
Horst Schubert.
\newblock {\em Categories}.
\newblock Springer-Verlag Berlin Heidelberg, 1972.
\newblock Translated from the German by Eva Gray.

\bibitem[vdPS97]{VDPS97}
Marius van~der Put and Michael~F. Singer.
\newblock {\em Galois theory of difference equations}, volume 1666 of {\em
  Lecture Notes in Math.}
\newblock Springer-Verlag, Berlin, 1997.

\end{thebibliography}

\end{document}